\documentclass[a4paper]{amsart}
\usepackage{amsmath,amssymb,amsthm}
\usepackage{appendix}

\theoremstyle{plain}
\newtheorem{theorem}{Theorem}[section]
\newtheorem{lemma}[theorem]{Lemma}

\theoremstyle{remark}
\newtheorem{remark}[theorem]{Remark}

\numberwithin{equation}{section}

\begin{document}

\title[Singular limits for incompressible fluids in moving thin domains]{On singular limit equations for incompressible fluids in moving thin domains}

\author{Tatsu-Hiko Miura}
\address{Graduate School of Mathematical Sciences, The University of Tokyo, 3-8-1 Komaba, Meguro, Tokyo, 153-8914 Japan}
\email{thmiura@ms.u-tokyo.ac.jp}



\begin{abstract}
  We consider the incompressible Euler and Navier-Stokes equations in a three-dimensional moving thin domain.
  Under the assumption that the moving thin domain degenerates into a two-dimensional moving closed surface as the width of the thin domain goes to zero, we give a heuristic derivation of singular limit equations on the degenerate moving surface of the Euler and Navier-Stokes equations in the moving thin domain and investigate relations between their energy structures.
  We also compare the limit equations with the Euler and Navier-Stokes equations on a stationary manifold, which are described in terms of the Levi-Civita connection.
\end{abstract}

\maketitle

\section{Introduction} \label{S:Introduction}
Fluid flows in a thin domain appear in many problems of natural sciences, e.g. ocean dynamics, geophysical fluid dynamics, and fluid flows in cell membranes.
In the study of the incompressible Navier-Stokes equations in a three-dimensional thin domain mathematical researchers are mainly interested in global existence of a strong solution for large data since a three-dimensional thin domain with sufficiently small width can be considered ``almost two-dimensional.''
It is also important to investigate the behavior of a solution as the width of a thin domain goes to zero.
We may naturally ask whether we can derive limit equations as a thin domain degenerates into a two-dimensional set and compare properties of solutions to the original three-dimensional equations and the corresponding two-dimensional limit equations.
There are several works studying such problems with a three-dimensional flat thin domain~\cite{HS10,IRS07,RS93,TZ96} of the form
\begin{align*}
  \Omega_\varepsilon = \{x=(x',x_3)\in\mathbb{R}^3 \mid x'\in\omega,~\varepsilon g_0(x')<x_3<\varepsilon g_1(x')\}
\end{align*}
for small $\varepsilon>0$, where $\omega$ is a two-dimensional domain and $g_0$ and $g_1$ are functions on $\omega$, and a three-dimensional thin spherical domain~\cite{TZ97} which is a region between two concentric spheres of near radii.
(We also refer to~\cite{R95} for the strategy of analysis of the Euler equations in a flat and spherical thin domain and its limit equations.)
However, mathematical studies of an incompressible fluid in a thin domain have not been done in the case where a thin domain and its degenerate set have more complicated geometric structures.
(See~\cite{PRR02} for the mathematical analysis of a reaction-diffusion equation in a thin domain degenerating into a lower dimensional manifold.)

In this paper we are concerned with the incompressible Euler and Navier-Stokes equations in a three-dimensional thin domain that moves in time.
The purpose of this paper is to give a heuristic derivation of singular limits of these equations as a moving thin domain degenerates into a two-dimensional moving closed surface.
We also investigate relations between the energy structures of the incompressible fluid systems in a moving thin domain and the corresponding limit systems on a moving closed surface.

Here let us explain our results on limit equations and strategy to derive them.
Let $\Gamma(t)$ be an evolving closed surface in $\mathbb{R}^3$ and $V_\Gamma^N(\cdot,t)$ and $\nu(\cdot,t)$ its (scalar) outward normal velocity and unit outward normal vector field, respectively.
We assume that $\Gamma(t)$ does not change its topology.
Also, let $\Omega_\varepsilon(t)$ be a tubular neighborhood of $\Gamma(t)$ of radius $\varepsilon$ in $\mathbb{R}^3$ with sufficiently small $\varepsilon>0$.
We consider the Euler equations
\begin{alignat}{2}
  \partial_tu+(u\cdot\nabla)u+\nabla p &= 0 &\quad &\text{in}\quad \Omega_\varepsilon(t),\,t\in(0,T), \label{E:Intro_Eu_Momentum} \\
  \mathrm{div}\,u &= 0 &\quad &\text{in}\quad \Omega_\varepsilon(t),\,t\in(0,T), \label{E:Intro_Eu_Continuity} \\
  u\cdot\nu_\varepsilon &= V_\varepsilon^N &\quad &\text{on}\quad \partial\Omega_\varepsilon(t),\,t\in(0,T) \label{E:Intro_Eu_Boundary}
\end{alignat}
and the Navier-Stokes equations with (perfect slip) Navier boundary condition
\begin{alignat}{2}
  \partial_tu+(u\cdot\nabla)u+\nabla p &= \mu_0\Delta u &\quad &\text{in}\quad \Omega_\varepsilon(t),\,t\in(0,T), \label{E:Intro_NS_Momentum}\\
  \mathrm{div}\,u &= 0 &\quad &\text{in}\quad \Omega_\varepsilon(t),\,t\in(0,T), \label{E:Intro_NS_Continuity} \\
  u\cdot\nu_\varepsilon &= V_\varepsilon^N &\quad &\text{on}\quad \partial\Omega_\varepsilon(t),\,t\in(0,T), \label{E:Intro_NS_B_Norm} \\
  [D(u)\nu_\varepsilon]_{\text{tan}} &= 0 &\quad &\text{on}\quad \partial\Omega_\varepsilon(t),\,t\in(0,T). \label{E:Intro_NS_B_Tan}
\end{alignat}
Here $\nu_\varepsilon$ and $V_\varepsilon^N$ denote the unit outward normal vector field and the (scaler) outward normal velocity of $\partial\Omega_\varepsilon(t)$.
Also, $\mu_0>0$ is the viscosity coefficient and $D(u):=\{\nabla u+(\nabla u)^T\}/2$ is the strain rate tensor with $(\nabla u)^T$ the transpose of the gradient matrix $\nabla u$.
We suppose that $\Omega_\varepsilon(t)$ admits the normal coordinate system $x=\pi(x,t)+d(x,t)\nu(\pi(x,t),t)$ for $x\in\Omega_\varepsilon(t)$, where $\pi(\cdot,t)$ is the closest point mapping onto $\Gamma(t)$ and $d(\cdot,t)$ is the signed distance from $\Gamma(t)$ increasing in the direction of $\nu(\cdot,t)$.
Based on the normal coordinates, we expand the velocity field $u(x,t)$ on $\Omega_\varepsilon(t)$ in powers of the signed distance $d(x,t)$ as
\begin{align} \label{E:Intro_Expand_U}
  u(x,t) = v(\pi(x,t),t)+d(x,t)v^1(\pi(x,t),t)+\cdots, \quad x\in\Omega_\varepsilon(t)
\end{align}
and the pressure $p(x,t)$ similarly.
We substitute them for the equations in $\Omega_\varepsilon(t)$ and determine equations on $\Gamma(t)$ that the zeroth order term $v$ in \eqref{E:Intro_Expand_U} satisfies.
Then we obtain limit equations of the Euler equations \eqref{E:Intro_Eu_Momentum}--\eqref{E:Intro_Eu_Boundary}:
\begin{alignat}{2}
  \partial^\bullet_vv+\nabla_\Gamma q+q^1\nu &= 0 &\quad&\text{on}\quad \Gamma(t),\,t\in(0,T), \label{E:Intro_Eu_Momentum_Limit} \\
  \mathrm{div}_\Gamma v &= 0 &\quad&\text{on}\quad \Gamma(t),\,t\in(0,T), \label{E:Intro_Eu_Continuity_Limit} \\
  v\cdot\nu &= V_\Gamma^N &\quad&\text{on}\quad \Gamma(t),\,t\in(0,T). \label{E:Intro_Eu_Normal_Limit}
\end{alignat}
Here $\partial^\bullet_v=\partial_t+v\cdot\nabla$ is the material derivative along the velocity field $v$ and $\nabla_\Gamma$ and $\mathrm{div}_\Gamma$ denote the tangential gradient and the surface divergence on $\Gamma(t)$, respectively (see Section~\ref{S:Preliminaries} for their definitions).
Similarly, we get limit equations of the Navier-Stokes equations \eqref{E:Intro_NS_Momentum}--\eqref{E:Intro_NS_B_Tan}:
\begin{alignat}{2}
  \partial^\bullet_vv+\nabla_\Gamma q+q^1\nu &= 2\mu_0\mathrm{div}_\Gamma(P_\Gamma D^{tan}(v)P_\Gamma) &\quad&\text{on}\quad \Gamma(t),\,t\in(0,T), \label{E:Intro_NS_Momentum_Limit}\\
  \mathrm{div}_\Gamma v &= 0 &\quad&\text{on}\quad \Gamma(t),\,t\in(0,T), \label{E:Intro_NS_Continuity_Limit} \\
  v\cdot\nu &= V_\Gamma^N &\quad&\text{on}\quad \Gamma(t),\,t\in(0,T). \label{E:Intro_NS_Normal_Limit}
\end{alignat}
Here $D^{tan}(v):=\{\nabla_\Gamma v+(\nabla_\Gamma v)^T\}/2$ and $P_\Gamma$ is the orthogonal projection onto the tangent plane of $\Gamma(t)$.
Note that if we take the average of \eqref{E:Intro_Expand_U} in the normal direction of $\Gamma(t)$ then
\begin{align*}
  \frac{1}{2\varepsilon}\int_{-\varepsilon}^\varepsilon u(y+\rho\nu(y,t),t)\,d\rho = v(y,t)+(\text{higher order terms in $\varepsilon$}), \quad y\in\Gamma(t).
\end{align*}
Therefore, formally speaking, our limit equations are equations satisfied by the limit of the average in the thin direction of a solution to the original Euler or Navier-Stokes equations in $\Omega_\varepsilon(t)$ as $\varepsilon$ goes to zero.
(The above method is also applied in~\cite{MGL17pre} to derive a limit equation of a nonlinear diffusion equation in a moving thin domain.)

In the equations \eqref{E:Intro_Eu_Momentum_Limit} and \eqref{E:Intro_NS_Momentum_Limit} the scalar function $q^1$, which comes from the normal derivative of the bulk pressure $p$ (see the expansion \eqref{E:P_Expand} of $p$ and \eqref{E:Nabla_P} in the proof of Theorem~\ref{T:Eu_Limit}), is determined by the normal component of \eqref{E:Intro_Eu_Momentum_Limit} and \eqref{E:Intro_NS_Momentum_Limit}.
Therefore, the limit Euler system \eqref{E:Intro_Eu_Momentum_Limit}--\eqref{E:Intro_Eu_Normal_Limit} is intrinsically equivalent to
\begin{align} \label{E:Intro_Eu_Limit_Tan}
  P_\Gamma\partial^\bullet_vv+\nabla_\Gamma q = 0, \quad \mathrm{div}_\Gamma v = 0, \quad v\cdot\nu = V_\Gamma^N
\end{align}
and the limit Navier-Stokes system \eqref{E:Intro_NS_Momentum_Limit}--\eqref{E:Intro_NS_Normal_Limit} is equivalent to
\begin{align} \label{E:Intro_NS_Limit_Tan}
  P_\Gamma\partial^\bullet_vv+\nabla_\Gamma q = 2\mu_0P_\Gamma\mathrm{div}_\Gamma(P_\Gamma D^{tan}(v)P_\Gamma), \quad \mathrm{div}_\Gamma v = 0, \quad v\cdot\nu = V_\Gamma^N.
\end{align}
We note that these tangential surface fluid systems were also derived in~\cite{JOR17pre,KLG16} recently.
The derivation of the Navier-Stokes equations on a moving surface in~\cite{JOR17pre} is based on local conservation laws of mass and linear momentum for a surface fluid.
On the other hand, the authors of~\cite{KLG16} applied a global energetic variational approach to derive several kinds of equations for an incompressible fluid on an evolving surface.

The viscous term $2\mu_0\mathrm{div}_\Gamma(P_\Gamma D^{tan}(v)P_\Gamma)$ in the momentum equation \eqref{E:Intro_NS_Momentum_Limit} of the limit Navier-Stokes system appears in the Boussinesq-Scriven surface fluid model which was first described by Boussinesq~\cite{B1913} and generalized by Scriven~\cite{S60} to an arbitrary curved moving surface (see also~\cite[Chapter~10]{Ari89bk} for derivation of the Boussinesq-Scriven surface fluid model).
In~\cite{ADe09} the Boussinesq-Scriven surface fluid model was considered to formulate a continuum model for fluid membranes in a bulk fluid, which contains equations for a viscous fluid on a curved moving surface, and study the effect of membrane viscosity in the dynamics of fluid membranes.
It was also studied in the context of two-phase flows~\cite{BGN15,BP10,NVW12} in which equations for a surface fluid are considered as the boundary condition on a fluid interface.

Since we consider an incompressible fluid on a moving surface or in its tubular neighborhood, some constraints on the motion of the surface are necessary.
For the existence of a surface incompressible fluid it is required that the area of the moving surface is preserved in time.
To consider a bulk incompressible fluid in the $\varepsilon$-tubular neighborhood of the moving surface for all $\varepsilon>0$ sufficiently small, we need another constraint on the moving surface besides the area preserving condition.
However, it is automatically satisfied by the Gauss-Bonnet theorem and the assumption that the moving surface does not change its topology.
See Remark~\ref{R:Incompressibility} for details.

When the surface does not move in time, our tangential limit system \eqref{E:Intro_Eu_Limit_Tan} of the Euler equations is the same as the Euler system on a fixed manifold derived by Arnol$'$d~\cite{A66,A89bk}, who applied the Lie group of diffeomorphisms of a manifold (see also Ebin and Marsden~\cite{EM70}).
Also, for a stationary surface our tangential limit system \eqref{E:Intro_NS_Limit_Tan} of the Navier-Stokes equations is the same as the Navier-Stokes system on a manifold derived by Taylor~\cite{T92}, although the authors of~\cite{KLG16} claim that \eqref{E:Intro_NS_Limit_Tan} is different from Taylor's system (see Remark~\ref{R:Deformation_Tensor}).
For detailed comparison of our limit systems and the systems derived in previous works see Remarks~\ref{R:Eu_Limit_Comparison} and~\ref{R:NS_Limit_Comparison}.
We further note that the function $q^1$ in the limit momentum equations \eqref{E:Intro_Eu_Momentum_Limit} and \eqref{E:Intro_NS_Momentum_Limit}, which is determined by the normal component of these equations, does not vanish even if the surface is stationary.
See Remarks~\ref{R:Eu_Limit_Comparison} and~\ref{R:NS_Limit_Comparison} for details.

Finally we note that our results are based on formal calculations and thus mathematical justification is required.
There are a few works that present rigorous derivation of limit equations in the case where a degenerate set is a hypersurface or a manifold.
Temam and Ziane~\cite{TZ97} derived limit equations for the Navier-Stokes equations in a thin spherical domain by characterizing the thin width limit of a solution to the original equations as a solution to the limit equations.
In~\cite{PRR02}, Prizzi, Rinaldi, and Rybakowski compared the dynamics of a reaction-diffusion equation in a thin domain and that of a limit equation when a thin domain degenerates into a lower dimensional manifold.
Recently, the present author derived a limit equation of the heat equation in a moving thin domain shrinking to a moving closed hypersurface by characterization of the thin width limit of a solution~\cite{M17}.
Although there are several tools and methods introduced in the above papers, it seems that mathematical justification of our results is difficult because of the nonlinearity of the equations and the evolution of the shape of the degenerate surface, and that we need some new techniques.

This paper is organized as follows.
In Section~\ref{S:Preliminaries} we give notations and formulas on quantities related to a moving surface and a moving thin domain.
In Sections~\ref{S:Limit_Eu} and~\ref{S:Limit_NS} we derive the limit equations of the Euler and Navier-Stokes equations in a moving thin domain, respectively.
In Section~\ref{S:Energy_Identity} we derive the energy identities of the Euler and Navier-Stokes equations and the corresponding limit equations and investigate relations between them.
In Appendices~\ref{S:Appendix_A} and~\ref{S:Appendix_B} we give proofs of lemmas in Section~\ref{S:Preliminaries} involving the differential geometry of a surface embedded in the Euclidean space.

\section{Preliminaries} \label{S:Preliminaries}
We fix notations on various quantities of a moving surface and give formulas on them.
All functions appearing in this section are assumed to be sufficiently smooth.

Lemmas in this section are proved by straightforward calculations.
To avoid making this section too long we give proofs of them in Appendix~\ref{S:Appendix_A}, except for the proofs of Lemmas~\ref{L:Riemannian_Connection} and~\ref{L:Compare_Viscous}.
Also, a proof of the formula \eqref{E:Bochner_Beltrami} in Lemma~\ref{L:Riemannian_Connection} is given in Appendix~\ref{S:Appendix_B}.
Although we are concerned with a two-dimensional surface in this paper, all notations and formulas in this section apply to hypersurfaces of any dimension with easy modifications.

\subsection{Moving surfaces and moving thin domains}
Let $\Gamma(t)$, $t\in[0,T]$ be a two-dimensional closed (i.e. compact and without boundary), connected, and oriented moving surface in $\mathbb{R}^3$.
The unit outward normal vector and the (scalar) outward normal velocity of $\Gamma(t)$ are denoted by $\nu(\cdot,t)$ and $V_\Gamma^N(\cdot,t)$, respectively.
Also, let $S_T:=\bigcup_{t\in(0,T)}\Gamma(t)\times\{t\}$ be a space-time hypersurface associated with $\Gamma(t)$.
We assume that $\Gamma(t)$ is smooth at each $t\in[0,T]$ and moves smoothly in time.
In particular, $\Gamma(t)$ does not change its topology.
By the smoothness assumption on $\Gamma(t)$, the (outward) principal curvatures $\kappa_1(\cdot,t)$ and $\kappa_2(\cdot,t)$ of $\Gamma(t)$ are bounded uniformly with respect to $t$.
Hence there is a tubular neighborhood
\begin{align*}
  N(t) := \{x\in\mathbb{R}^3 \mid \mathrm{dist}(x,\Gamma(t)) < \delta\}
\end{align*}
of radius $\delta>0$ independent of $t$ that admits the normal coordinate system
\begin{align} \label{E:N_Coordinate}
  x = \pi(x,t)+d(x,t)\nu(\pi(x,t),t), \quad x\in N(t),
\end{align}
where $\pi(\cdot,t)$ is the closest point mapping onto $\Gamma(t)$ and $d(\cdot,t)$ is the signed distance function from $\Gamma(t)$ (see e.g.~\cite[Lemma~2.8]{DE13}).
Moreover, the mapping $\pi$ and the signed distance $d$ are smooth in the closure (in $\mathbb{R}^4$) of a space-time noncylindrical domain $N_T := \bigcup_{t\in(0,T)}N(t)\times\{t\}$.
We assume that $d(\cdot,t)$ increases in the direction of $\nu(\cdot,t)$.
Therefore,
\begin{align}
  \nabla d(x,t) &= \nu(\pi(x,t),t), \quad (x,t) \in N_T, \label{E:Nabla_Dist}\\
  \partial_td(y,t) &= -V_\Gamma^N(y,t), \quad (y,t) \in S_T. \label{E:Dt_Dist_Surface}
\end{align}
Moreover, differentiating both sides of
\begin{align*}
  d(x,t) = \{x-\pi(x,t)\}\cdot\nabla d(x,t), \quad d(\pi(x,t),t) = 0
\end{align*}
with respect to $t$ and using \eqref{E:Nabla_Dist} and \eqref{E:Dt_Dist_Surface} we easily get
\begin{align} \label{E:Dt_Dist_NBD}
  \partial_td(x,t) = \partial_td(\pi(x,t),t) = -V_\Gamma^N(\pi(x,t),t), \quad (x,t)\in N_T.
\end{align}
For a sufficiently small $\varepsilon>0$ we define a moving thin domain $\Omega_\varepsilon(t)$ in $\mathbb{R}^3$ as
\begin{align*}
  \Omega_\varepsilon(t) := \{x\in\mathbb{R}^3 \mid \mathrm{dist}(x,\Gamma(t))<\varepsilon\}
\end{align*}
and a space-time noncylindrical domain $Q_{\varepsilon,T}$ and its lateral boundary $\partial_\ell Q_{\varepsilon,T}$ as
\begin{align*}
  Q_{\varepsilon,T} := \bigcup_{t\in(0,T)}\Omega_\varepsilon(t)\times\{t\}, \quad \partial_\ell Q_{\varepsilon,T} := \bigcup_{t\in(0,T)}\partial\Omega_\varepsilon(t)\times\{t\}.
\end{align*}
Since $\Omega_\varepsilon(t)$ is a tubular neighborhood of $\Gamma(t)$, the unit outward normal vector $\nu_\varepsilon(\cdot,t)$ and the outward normal velocity $V_\varepsilon^N(\cdot,t)$ of its boundary are given by
\begin{align}
  \begin{split} \label{E:N_Vect_Dom}
    \nu_\varepsilon(x,t) &=
    \begin{cases}
      \nu(\pi(x,t),t) &\text{if}\quad d(x,t) = \varepsilon, \\
      -\nu(\pi(x,t),t) &\text{if}\quad d(x,t) = -\varepsilon,
    \end{cases}
  \end{split} \\
  \begin{split} \label{E:N_Vel_Dom}
    V_\varepsilon^N(x,t) &=
    \begin{cases}
      V_\Gamma^N(\pi(x,t),t) &\text{if}\quad d(x,t) = \varepsilon, \\
      -V_\Gamma^N(\pi(x,t),t) &\text{if}\quad d(x,t) = -\varepsilon.
    \end{cases}
  \end{split}
\end{align}

\subsection{Notations and formulas for quantities on fixed surfaces}
In this subsection we fix and suppress the time $t\in[0,T]$.
Hence $\Gamma$ denotes a two-dimensional closed, connected, oriented and smooth surface in $\mathbb{R}^3$.
Let us give notations and formulas for several quantities on the fixed surface $\Gamma$.
(In the sequel we use the same notations given in this subsection for the moving surface $\Gamma(t)$.)
Let $P_\Gamma$ be the orthogonal projection onto the tangent plane of $\Gamma$ at each point on $\Gamma$ given by
\begin{align*}
  P_\Gamma(y) := I_3-\nu(y)\otimes\nu(y), \quad y\in\Gamma,
\end{align*}
where $I_3$ is the identity matrix of three dimension and $a\otimes b$ for $a,b\in\mathbb{R}^3$ denotes the tensor product of $a$ and $b$ given by
\begin{align*}
  a\otimes b :=
  \begin{pmatrix}
    a_1b_1 & a_1b_2 & a_1b_3 \\
    a_2b_1 & a_2b_2 & a_2b_3 \\
    a_3b_1 & a_3b_2 & a_3b_3
  \end{pmatrix}, \quad a=(a_1,a_2,a_3),\,b=(b_1,b_2,b_3).
\end{align*}
For a function $f$ on $\Gamma$ we define its tangential gradient $\nabla_\Gamma f$ as
\begin{align*}
  \nabla_\Gamma f(y) := P_\Gamma(y)\nabla\tilde{f}(y), \quad y\in\Gamma.
\end{align*}
Here $\tilde{f}$ is an extension of $f$ to $N$ satisfying $\tilde{f}|_\Gamma=f$.
Note that the tangential gradient of $f$ is independent of the choice of its extension (see e.g.~\cite[Lemma~2.4]{DE13}).
Also, it is easy to see that $\nabla_\Gamma f\cdot\nu = 0$ and $P_\Gamma\nabla_\Gamma f=\nabla_\Gamma f$ hold on $\Gamma$.
The tangential derivative operators are given by
\begin{align*}
  \partial_i^{tan} f(y) := \sum_{j=1}^3\{\delta_{ij}-\nu_i(y)\nu_j(y)\}\partial_j\tilde{f}(y), \quad i=1,2,3
\end{align*}
so that $\nabla_\Gamma=(\partial_1^{tan},\partial_2^{tan},\partial_3^{tan})$, which are again independent of the choice of an extension $\tilde{f}$ of $f$.
For example, we may take the constant extension in the normal direction of $\Gamma$ given by $\bar{f}(x):=f(\pi(x))$ for $x\in N$.

For vector fields $F=(F_1,F_2,F_3)$ on $N$ and $G=(G_1,G_2,G_3)$ on $\Gamma$, we define the gradient matrix and the divergence of $F$ as
\begin{align*}
  \nabla F :=
  \begin{pmatrix}
    \partial_1F_1 & \partial_1F_2 & \partial_1F_3 \\
    \partial_2F_1 & \partial_2F_2 & \partial_2F_3 \\
    \partial_3F_1 & \partial_3F_2 & \partial_3F_3
  \end{pmatrix}, \quad
  \mathrm{div}\,F := \sum_{i=1}^3\partial_iF_i
\end{align*}
and the tangential gradient matrix and the surface divergence of $G$ as
\begin{align*}
  \nabla_\Gamma G :=
  \begin{pmatrix}
    \partial_1^{tan}G_1 & \partial_1^{tan}G_2 & \partial_1^{tan}G_3 \\
    \partial_2^{tan}G_1 & \partial_2^{tan}G_2 & \partial_2^{tan}G_3 \\
    \partial_3^{tan}G_1 & \partial_3^{tan}G_2 & \partial_3^{tan}G_3
  \end{pmatrix}, \quad
  \mathrm{div}_\Gamma G := \sum_{i=1}^3\partial_i^{tan}G_i.
\end{align*}
These notations are consistent with the formula $\nabla_\Gamma G=P_\Gamma\nabla\widetilde{G}$ on $\Gamma$, where $\widetilde{G}$ is an arbitrary extension of $G$ to $N$ with $\widetilde{G}|_\Gamma=G$.
For a function $f$ on $\Gamma$ we denote by $\nabla_\Gamma^2f$ the tangential Hessian matrix of $f$ whose $(i,j)$-entry is given by $\partial_i^{tan}\partial_j^{tan}f$ ($i,j=1,2,3$).
Let $M$ be a $3\times3$ matrix-valued function defined on $N$ or on $\Gamma$ of the form
\begin{align*}
  M = (M_{ij})_{i,j} =
  \begin{pmatrix}
    M_{11} & M_{12} & M_{13} \\
    M_{21} & M_{22} & M_{23} \\
    M_{31} & M_{32} & M_{33}
  \end{pmatrix}.
\end{align*}
We define the divergence $\mathrm{div}\,M$ on $N$ or the surface divergence $\mathrm{div}_\Gamma M$ on $\Gamma$ as a vector field whose $j$-th component is given by
\begin{align*}
  [\mathrm{div}\,M]_j := \sum_{i=1}^3\partial_iM_{ij} \quad\text{or}\quad [\mathrm{div}_\Gamma M]_j := \sum_{i=1}^3\partial_i^{tan}M_{ij}, \quad j=1,2,3.
\end{align*}
Finally we set
\begin{gather*}
  A := -\nabla_\Gamma\nu = (-\partial_i^{tan}\nu_j)_{i,j},\quad \Delta_\Gamma := \mathrm{div}_\Gamma\nabla_\Gamma = \sum_{i=1}^3(\partial_i^{tan})^2, \\
  H := -\mathrm{div}_\Gamma\nu =  \mathrm{tr}[A], \quad K := \kappa_1\kappa_2
\end{gather*}
and call them the Weingarten map of $\Gamma$, the Laplace-Beltrami operator on $\Gamma$, (twice) the mean curvature of $\Gamma$, and the Gaussian curvature of $\Gamma$, respectively.
The usual Laplacian $\Delta$ and the Laplace-Beltrami operator $\Delta_\Gamma$ acting on vector fields are understood to be componentwise operators.

\begin{lemma} \label{L:Weingarten}
  For all $y\in \Gamma$ we have
  \begin{gather}
    A(y)\nu(y) = 0, \label{E:W_Normal} \\
    A(y)P_\Gamma(y) = P_\Gamma(y)A(y) = A(y), \label{E:W_OrthProj} \\
    A(y) = -\nabla^2d(y). \label{E:W_Hessian}
  \end{gather}
\end{lemma}

By \eqref{E:W_Normal} we see that $A$ has the eigenvalue $0$.
Note that the other eigenvalues of $A$ are $\kappa_1$ and $\kappa_2$ (see e.g.~\cite[Section VII.5]{KN96bk}) and thus
\begin{align} \label{E:Sum_Curvature}
  H(y) = \kappa_1(y)+\kappa_2(y), \quad y\in\Gamma.
\end{align}
Also, $A$ is symmetric (i.e. $\partial_i^{tan}\nu_j=\partial_j^{tan}\nu_i$) and $H=-\Delta d$ holds on $\Gamma$ by \eqref{E:W_Hessian}.

The tangential derivatives $\partial_i^{tan}$ ($i=1,2,3$) are noncommutative in general.
An exchange formula for them includes the unit outward normal of the surface.

\begin{lemma} \label{L:Exchange_Derivative}
  Let $f$ be a function on $\Gamma$.
  For each $i,j=1,2,3$ we have
  \begin{align} \label{E:Exchange_Derivative}
    \partial_i^{tan}\partial_j^{tan}f-\partial_j^{tan}\partial_i^{tan}f = [A\nabla_\Gamma f]_i\nu_j-[A\nabla_\Gamma f]_j\nu_i.
  \end{align}
  Here $[A\nabla_\Gamma f]_i$ denotes the $i$-th component of the vector field $A\nabla_\Gamma f$.
\end{lemma}

The next formula is a consequence of \eqref{E:Exchange_Derivative}, which we use in Section~\ref{S:Limit_NS} to express a viscous term of limit equations of the Navier-Stokes equations in terms of the Laplace-Beltrami operator.
For a vector field $v$ on $\Gamma$ we set
\begin{align} \label{E:Tan_Strain}
  D^{tan}(v) := \frac{\nabla_\Gamma v+(\nabla_\Gamma v)^T}{2}.
\end{align}
The matrices $D^{tan}(v)$ and $P_\Gamma D^{tan}(v)P_\Gamma$ are called a tangential strain rate and a projected strain rate in~\cite{KLG16}, respectively.

\begin{lemma} \label{L:Viscous_General}
  Let $v$ be a (not necessarily tangential) vector field on $\Gamma$.
  Then
  \begin{align} \label{E:Viscous_General}
    2\mathrm{div}_\Gamma(P_\Gamma D^{tan}(v)P_\Gamma) = 2\mathrm{tr}[A\nabla_\Gamma v]\nu+P_\Gamma(\Delta_\Gamma v)+\nabla_\Gamma(\mathrm{div}_\Gamma v)+H(\nabla_\Gamma v)\nu
  \end{align}
  holds on $\Gamma$ (note that $(\nabla_\Gamma v)\nu=P_\Gamma(\nabla_\Gamma v)\nu$ on the right-hand side is tangential).
\end{lemma}

To compare our limit systems with the incompressible fluid systems on a fixed manifold derived by Arnol$'$d~\cite{A66,A89bk} and Taylor~\cite{T92} we need formulas on the Levi-Civita connection.
Let $\overline{\nabla}$ be the Levi-Civita connection on $\Gamma$ with respect to the metric on $\Gamma$ induced by the Euclidean metric of $\mathbb{R}^3$ (see e.g.~\cite[Section~2.3]{C15bk} and~\cite[Sections~3.3.1 and~4.1.2]{N07bk} for the definition of the Levi-Civita connection).
Hence for tangential vector fields $X$ and $Y$ on $\Gamma$ the covariant derivative of $X$ along $Y$ is denoted by $\overline{\nabla}_YX$, which is again a tangential vector field on $\Gamma$.
The Levi-Civita connection is considered as a mapping
\begin{align*}
  \overline{\nabla}\colon C^\infty(T\Gamma) \to C^\infty(T^\ast\Gamma\otimes T\Gamma), \quad X \mapsto \overline{\nabla} X,
\end{align*}
where $T\Gamma$ and $T^\ast\Gamma$ are the tangent and cotangent bundle of $\Gamma$, respectively, and for a vector bundle $E$ over $\Gamma$ we denote by $C^\infty(E)$ the set of all smooth sections of $E$.
(Hence $C^\infty(T\Gamma)$ denotes the set of all smooth tangential vector fields on $\Gamma$.
We refer to~\cite[Chapter~10]{L13bk} for the definitions of a vector bundle and a section.)
Also, for a tangential vector field $X$ on $\Gamma$ the notation $\overline{\nabla}X$ stands for a mapping $Y\mapsto\overline{\nabla}_YX$ from $C^\infty(T\Gamma)$ into itself.
Then we write $\overline{\nabla}^\ast\colon C^\infty(T^\ast\Gamma\otimes T\Gamma)\to C^\infty(T\Gamma)$ for the formal adjoint operator of $\overline{\nabla}$ (see~\cite[Section~10.1.3]{N07bk}) and set $\Delta_B := -\overline{\nabla}^\ast\overline{\nabla}$.
The operator $\Delta_B\colon C^\infty(T\Gamma)\to C^\infty(T\Gamma)$ is called the Bochner Laplacian (note that there is another definition of the Bochner Laplacian where the sign is taken opposite).

\begin{lemma} \label{L:Riemannian_Connection}
  Let $X$ and $Y$ are tangential vector fields on $\Gamma$.
  Then
  \begin{align}
    (Y\cdot\nabla)\widetilde{X} &= \overline{\nabla}_YX+(AX\cdot Y)\nu, \label{E:Gauss_Formula}\\
    \Delta_BX &= P_\Gamma(\Delta_\Gamma X)+A^2X \label{E:Bochner_Beltrami}
  \end{align}
  hold on $\Gamma$.
  Here $\widetilde{X}$ is an extension of $X$ to $N$ with $\widetilde{X}|_\Gamma=X$ and $(Y\cdot\nabla)\widetilde{X}$ denotes the directional derivative of $\widetilde{X}$ along $Y$ in $\mathbb{R}^3$, i.e.
  \begin{align*}
    (Y\cdot\nabla)\widetilde{X} = \left(\sum_{i=1}^3Y_i\partial_i\widetilde{X}_1,\,\sum_{i=1}^3Y_i\partial_i\widetilde{X}_2,\,\sum_{i=1}^3Y_i\partial_i\widetilde{X}_3\right).
  \end{align*}
  Also, the left-hand side of \eqref{E:Gauss_Formula} is independent of the choice of the extension $\widetilde{X}$.
\end{lemma}

The formula \eqref{E:Gauss_Formula} is well-known as the Gauss formula (see e.g.~\cite[Section~4.2]{C15bk} and~\cite[Section~VII.3]{KN96bk}) and we omit its proof.
Note that $(Y\cdot\nabla)\widetilde{X}=(Y\cdot\nabla_\Gamma)X$ on $\Gamma$ since $Y$ is tangential.
Hence the Gauss formula \eqref{E:Gauss_Formula} is also expressed as
\begin{align} \label{E:Gauss_Formula_Another}
  (Y\cdot\nabla_\Gamma)X = \overline{\nabla}_YX+(AX\cdot Y)\nu \quad\text{on}\quad \Gamma
\end{align}
for tangential vector fields $X$ and $Y$ on $\Gamma$.
We also call \eqref{E:Gauss_Formula_Another} the Gauss formula.

A proof of the formula \eqref{E:Bochner_Beltrami} is given in Appendix~\ref{S:Appendix_B}.
Note that \eqref{E:Bochner_Beltrami} is useful by itself since it gives a global expression under the fixed Cartesian coordinate system of the Bochner Laplacian acting on tangential vector fields on $\Gamma$, which is originally defined intrinsically and represented under only local coordinate systems.

Combining Lemmas~\ref{L:Viscous_General} and~\ref{L:Riemannian_Connection} we get the following formula on the surface divergence  of the projected strain rate, which is crucial for comparison of our limit Navier-Stokes system and the incompressible viscous fluid system on a manifold derived by Taylor~\cite{T92} (see Remark~\ref{R:NS_Limit_Comparison}).

\begin{lemma} \label{L:Compare_Viscous}
  For a tangential vector field $v$ on $\Gamma$ satisfying $\mathrm{div}_\Gamma v=0$ we have
  \begin{align} \label{E:Compare_Viscous}
    2P_\Gamma\mathrm{div}_\Gamma(P_\Gamma D^{tan}(v)P_\Gamma) = \Delta_Bv+Kv \quad\text{on}\quad \Gamma.
  \end{align}
\end{lemma}

\begin{proof}
  Let $v$ be a tangential vector field on $\Gamma$ satisfying $\mathrm{div}_\Gamma v=0$.
  Then
  \begin{align*}
    (\nabla_\Gamma v)\nu = \nabla_\Gamma(v\cdot\nu)-(\nabla_\Gamma\nu)v = Av
  \end{align*}
  by $v\cdot\nu=0$ and $-\nabla_\Gamma\nu=A$.
  Applying this and
  \begin{align*}
    P_\Gamma(\mathrm{tr}[A\nabla_\Gamma v]\nu) = \mathrm{tr}[A\nabla_\Gamma v]P_\Gamma\nu = 0, \quad \mathrm{div}_\Gamma v = 0
  \end{align*}
  to the formula \eqref{E:Viscous_General}, and observing that $(\nabla_\Gamma v)\nu=Av$ is tangential, we have
  \begin{align} \label{Pf1:Compare_Viscous}
    2P_\Gamma\mathrm{div}_\Gamma(P_\Gamma D^{tan}(v)P_\Gamma) = P_\Gamma(\Delta_\Gamma v)+HAv.
  \end{align}
  Moreover, since $A$ is symmetric and has the eigenvalues $0$, $\kappa_1$, and $\kappa_2$, where the eigenvector corresponding to the eigenvalue $0$ is $\nu$ (see Lemma~\ref{L:Weingarten}), for each $y\in\Gamma$ we can take an orthonormal basis $\{e_1,e_2\}$ of the tangent plane of $\Gamma$ at $y$ such that $Ae_i=\kappa_ie_i$, $i=1,2$.
  (The vectors $e_1$ and $e_2$ are called the principal directions at $y$. See e.g.~\cite[Section VII.5]{KN96bk} for details.)
  Expressing the tangential vector $v$ as a linear combination of $e_1$ and $e_2$ and using $H=\kappa_1+\kappa_2$ and $K=\kappa_1\kappa_2$ we easily obtain $HAv = Kv+A^2v$.
  Applying this and \eqref{E:Bochner_Beltrami} to \eqref{Pf1:Compare_Viscous} we obtain \eqref{E:Compare_Viscous}.
\end{proof}

Besides derivation of limit equations, we are also interested in thin width limits of energy identities for the Euler and Navier-Stokes equations.
To derive limit energy identities we give change of variables formulas for integrals over level-set surfaces and tubular neighborhoods of $\Gamma$.
For $y\in\Gamma$ and $\rho\in[-\varepsilon,\varepsilon]$ we set
\begin{align} \label{E:Jacobian}
  J(y,\rho) := \{1-\rho\kappa_1(y)\}\{1-\rho\kappa_2(y)\} = 1-\rho H(y)+\rho^2 K(y).
\end{align}
Here the second equality follows from the definition of the Gaussian curvature and \eqref{E:Sum_Curvature}.
The function $J$ is the Jacobian appearing in the following change of variables formulas (see~\cite[Section 14.6]{GT01bk} or Appendix~\ref{S:Appendix_A}).

\begin{lemma} \label{L:Co-area}
  For a function $f$ on $\overline{\Omega_\varepsilon}$ we have
  \begin{align} \label{E:Co-area_NBD}
    \int_{\Omega_\varepsilon}f(x)\,dx = \int_{\Gamma}\int_{-\varepsilon}^\varepsilon f(y+\rho\nu(y))J(y,\rho)\,d\rho\,d\mathcal{H}^2(y)
  \end{align}
  and
  \begin{multline} \label{E:Co-area_BDY}
    \int_{\partial\Omega_\varepsilon}f(x)\,d\mathcal{H}^2(x) = \int_{\Gamma}f(y+\varepsilon\nu(y))J(y,\varepsilon)\,d\mathcal{H}^2(y) \\
    +\int_{\Gamma}f(y-\varepsilon\nu(y))J(y,-\varepsilon)\,d\mathcal{H}^2(y).
  \end{multline}
  Here $\mathcal{H}^2$ denotes the two-dimensional Hausdorff measure.
\end{lemma}

When we use Lemma~\ref{L:Co-area} with the moving surface $\Gamma(t)$ we write $J(y,t,\rho)$ for the Jacobian given by \eqref{E:Jacobian}.

\subsection{Material derivatives and differentiation of composite functions with the closest point mapping}
Now let us return to the moving surface $\Gamma(t)$.
We first give a material time derivative of a function on $S_T$.
Let $v$ be a vector field on $S_T$ with $v\cdot\nu=V_\Gamma^N$.
Suppose that there exists the flow map $\Phi_v$ of $v$, i.e. $\Phi_v(\cdot,t)\colon\Gamma(0)\to\mathbb{R}^3$ is a diffeomorphism onto its range for each $t\in[0,T]$ and
\begin{align*}
  \Phi_v(Y,0) = Y, \quad \frac{d\Phi_v}{dt}(Y,t) = v(\Phi_v(Y,t),t) \quad\text{for}\quad (Y,t) \in \Gamma(0)\times(0,T).
\end{align*}
Note that $\Phi_v(\cdot,t)$ is a diffeomorphism from $\Gamma(0)$ onto $\Phi_v(\Gamma(0),t)=\Gamma(t)$ for each $t\in[0,T]$ since the normal component of $v$ is equal to the outward normal velocity $V_\Gamma^N$ of the moving surface $\Gamma(t)$, which completely determines the change of the shape of $\Gamma(t)$.
We define the material derivative of a function $f$ on $S_T$ along the velocity field $v$ as
\begin{align*}
  \partial^\bullet_vf(\Phi_v(Y,t),t) := \frac{d}{dt}\bigl(f(\Phi_v(Y,t),t)\bigr), \quad (Y,t)\in\Gamma(0)\times(0,T).
\end{align*}
By the chain rule of differentiation it is also represented as
\begin{align} \label{E:Material_Euler_Coordinate}
  \partial^\bullet_vf(y,t) = \partial_t\tilde{f}(y,t)+v(y,t)\cdot\nabla\tilde{f}(y,t), \quad (y,t) \in S_T,
\end{align}
where $\tilde{f}$ is an arbitrary extension of $f$ to $N_T$ satisfying $\tilde{f}|_{S_T}=f$.
We write $\partial^\circ$ for $\partial^\bullet_v$ with $v=V_\Gamma^N\nu$ and call it the normal time derivative.
Note that the normal time derivative of a function $f$ on $S_T$ is equal to the time derivative of its constant extension $\bar{f}$ in the normal direction, i.e.
\begin{align*}
  \partial^\circ f(y,t) = \partial_t\bar{f}(y,t) = \frac{d}{dt}\bigl(f(\pi(y,t),t)\bigr), \quad (y,t) \in S_T.
\end{align*}
Also, for a tangential vector field $v^T$ on $S_T$ the material derivative of $f$ along the velocity field of the form $v=V_\Gamma^N\nu+v^T$ is expressed as
\begin{align} \label{E:Material_and_Normal}
  \partial^\bullet_vf = \partial^\circ f+v^T\cdot\nabla_\Gamma f \quad\text{on}\quad S_T
\end{align}
by \eqref{E:Material_Euler_Coordinate} and $v^T\cdot\nabla\tilde{f}=v^T\cdot\nabla_\Gamma f$ on $S_T$ since $v^T$ is tangential.
See also~\cite[Section~3]{CFG05} for the time derivative of functions on a moving surface.

In the following sections we frequently differentiate the composition of a function on $\Gamma(t)$ and the closest point mapping $\pi(\cdot,t)$.
To avoid repetition of the same calculations we give several formulas on derivatives of composite functions with $\pi$.

Let $f(x,t)$ be a function on $Q_{\varepsilon,T}$.
Based on the normal coordinate system $x=\pi(x,t)+d(x,t)\nu(\pi(x,t),t)$ for $x\in\Omega_\varepsilon(t)$, we expand $f(x,t)$ in powers of the signed distance $d(x,t)$:
\begin{align*}
  f(x,t) = g(\pi(x,t),t)+d(x,t)g^1(\pi(x,t),t)+\cdots.
\end{align*}
Here $g$, $g^1$, and the coefficients of higher order terms in $d(x,t)$ are considered as functions on $S_T$.
Also, for $k\in\mathbb{N}$ we write $R(d(x,t)^k)$ for the terms of order higher than $k-1$ with respect to small $d(x,t)$, i.e.
\begin{gather} \label{E:Residual_Function}
  \begin{gathered}
    f(x,t) = g(\pi(x,t),t)+\cdots+d(x,t)^{k-1}g^{k-1}(\pi(x,t),t)+R(d(x,t)^k), \\
    R(d(x,t)^k) = d(x,t)^kg^k(\pi(x,t),t)+d(x,t)^{k+1}g^{k+1}(\pi(x,t),t)+\cdots.
  \end{gathered}
\end{gather}
In the sequel, we also use Landau's symbol $O(\varepsilon^k)$ (as $\varepsilon\to0$) for a nonnegative integer $k$, i.e. $O(\varepsilon^k)$ is a quantity satisfying $|O(\varepsilon^k)|\leq C\varepsilon^k$ for small $\varepsilon>0$ with a constant $C>0$ independent of $\varepsilon$.
Note that, contrary to $O(\varepsilon^k)$, we may differentiate $R(d(x,t)^k)$ with respect to $x$ and $t$ since it just stands for the higher order terms in the expansion \eqref{E:Residual_Function} with respect to small $d(x,t)$, and the $l$-th order derivative of $R(d(x,t)^k)$ is $R(d(x,t)^{k-l})$ for $l\leq k$.
Also, $R(d(x,t)^k)=O(\varepsilon^k)$ for $(x,t)\in Q_{\varepsilon,T}$ and $k\in\mathbb{N}$ by $|d(x,t)|<\varepsilon$ on $Q_{\varepsilon,T}$.
We use the same notations on the expansion \eqref{E:Residual_Function} for functions on $\Omega_\varepsilon(t)$ with each fixed $t\in[0,T]$.

\begin{lemma} \label{L:Diff_Composite}
  Let $f$ be a scalar- or vector-valued function on $S_T$.
  The derivatives of the composite function $f(\pi(x,t),t)$ with respect to $x$ and $t$ are of the form
  \begin{align}
    \nabla\bigl(f(\pi,t)\bigr) &= \nabla_\Gamma f(\pi,t)+d(x,t)[A\nabla_\Gamma f](\pi,t)+R(d(x,t)^2), \label{E:Nabla_Composite} \\
    \partial_t\bigl(f(\pi,t)\bigr) &= \partial^\circ f(\pi,t)+d(x,t)[(\nabla_\Gamma V_\Gamma^N\cdot\nabla_\Gamma)f](\pi,t)+R(d(x,t)^2) \label{E:Dt_Composite}
  \end{align}
  for $(x,t)\in Q_{\varepsilon,T}$.
  Here we abbreviate $\pi(x,t)$ to $\pi$.
\end{lemma}

We also give an expansion formula for the divergence of a matrix-valued function which we need to derive limit equations of the Navier-Stokes equations.

\begin{lemma} \label{L:Div_Matrix}
  Let $S$ and $S^1$ be $3\times 3$ matrix-valued functions on $\Gamma(t)$ with each fixed $t\in(0,T)$.
  For $x\in \Omega_\varepsilon(t)$ we set
  \begin{align*}
    D(x) = S(\pi(x,t))+d(x,t)S^1(\pi(x,t))+R(d(x,t)^2).
  \end{align*}
  Then we have
  \begin{align} \label{E:Div_Matrix}
    \mathrm{div}\,D(x) = \mathrm{div}_\Gamma S(\pi(x,t))+\bigl(S^1(\pi(x,t))\bigr)^T\nu(\pi,t)+R(d(x,t)).
  \end{align}
  for $x\in \Omega_\varepsilon(t)$.
  Here $(S^1)^T$ denotes the transpose of the matrix $S^1$.
\end{lemma}

\section{Limit equations of the Euler equations} \label{S:Limit_Eu}
We consider the incompressible Euler equations in $\Omega_\varepsilon(t)$:
\begin{alignat}{2}
  \partial u_t+(u\cdot\nabla)u+\nabla p &= 0 &\quad &\text{in}\quad Q_{\varepsilon,T}, \label{E:Eu_Momentum} \\
  \mathrm{div}\,u &= 0 &\quad &\text{in}\quad Q_{\varepsilon,T}, \label{E:Eu_Continuity} \\
  u\cdot\nu_\varepsilon &= V_\varepsilon^N &\quad &\text{on}\quad \partial_\ell Q_{\varepsilon,T}. \label{E:Eu_Boundary}
\end{alignat}
Here $u=(u_1,u_2,u_3)$ is the velocity of a bulk fluid and $p$ is the pressure.
The goal of this section is to derive limit equations of the Euler equations as $\varepsilon$ goes to zero.
According to the normal coordinate system \eqref{E:N_Coordinate}, we expand $u$ and $p$ with respect to the signed distance $d(x,t)$ as
\begin{align}
  u(x,t) &= v(\pi(x,t),t)+d(x,t)v^1(\pi(x,t),t)+R(d(x,t)^2), \label{E:Eu_U_Expand}\\
  p(x,t) &= q(\pi(x,t),t)+d(x,t)q^1(\pi(x,t),t)+R(d(x,t)^2). \label{E:P_Expand}
\end{align}
Here we used the notation \eqref{E:Residual_Function}.
The limit equations are given as the principal term in the expansion with respect to $d(x,t)$ of the Euler equations in $\Omega_\varepsilon(t)$.

\begin{theorem} \label{T:Eu_Limit}
  Let $u$ and $p$ satisfy the Euler equations \eqref{E:Eu_Momentum}--\eqref{E:Eu_Boundary} in the moving thin domain $\Omega_\varepsilon(t)$.
  Then the normal component of the zeroth order term $v$ in the expansion \eqref{E:Eu_U_Expand} is equal to the outward normal velocity of the moving surface $\Gamma(t)$, i.e. $v\cdot\nu=V_\Gamma^N$.
  Moreover, $v$ and the zeroth order term $q$ and the first order term $q^1$ in the expansion \eqref{E:P_Expand} satisfy
  \begin{align}
    \partial^\bullet_vv+\nabla_\Gamma q+q^1\nu &= 0 \quad\text{on}\quad S_T, \label{E:Eu_Momentum_Limit} \\
    \mathrm{div}_\Gamma v &= 0 \quad\text{on}\quad S_T. \label{E:Eu_Continuity_Limit}
  \end{align}
\end{theorem}

Before starting to prove Theorem~\ref{T:Eu_Limit} we give remarks on the limit equations \eqref{E:Eu_Momentum_Limit}--\eqref{E:Eu_Continuity_Limit} and necessary conditions on the motion of $\Gamma(t)$ for the existence of incompressible fluids in $\Gamma(t)$ and $\Omega_\varepsilon(t)$ for all $\varepsilon>0$.

\begin{remark} \label{R:Eu_Limit_Comparison}
  Let us explain how the limit equations \eqref{E:Eu_Momentum_Limit} and \eqref{E:Eu_Continuity_Limit} determine $v$, $q$, and $q^1$.
  As stated in Theorem~\ref{T:Eu_Limit}, the normal component of $v$ is equal to the outward normal velocity of the moving surface.
  The tangential component of $v$ and the scalar function $q$ are determined by the equations
  \begin{align} \label{E:Eu_Limit_Tangential}
    P_\Gamma\partial^\bullet_vv+\nabla_\Gamma q = 0, \quad \mathrm{div}_\Gamma v = 0 \quad\text{on}\quad S_T.
  \end{align}
  Finally the scalar function $q^1$ is given just by the inner product of \eqref{E:Eu_Momentum_Limit} and $\nu$:
  \begin{align} \label{E:Eu_Limit_Normal_Move}
    q^1 = -\partial^\bullet_vv\cdot\nu \quad\text{on}\quad S_T.
  \end{align}
  Note that $q^1$ comes from the normal derivative of the pressure $p$ of the bulk fluid in the moving thin domain (see \eqref{E:Nabla_P} below).

  The system \eqref{E:Eu_Limit_Tangential} is the same as the incompressible Euler system (II) in~\cite{KLG16} with the constant density.
  When the surface $\Gamma(t)=\Gamma$ is stationary, the limit velocity $v$ is tangential ($v\cdot\nu=V_\Gamma^N=0$) and $P_\Gamma\{(v\cdot\nabla)v\}=\overline{\nabla}_vv$ holds on $\Gamma$ by the Gauss formula \eqref{E:Gauss_Formula}, where $\overline{\nabla}_vv$ is the covariant derivative.
  From this and the fact that $P_\Gamma$ is independent of the time it follows that
  \begin{align} \label{E:Projection_Material_Stationary}
    P_\Gamma\partial^\bullet_vv = P_\Gamma\partial_tv+P_\Gamma\{(v\cdot\nabla)v\} = \partial_tv+\overline{\nabla}_vv \quad\text{on}\quad \Gamma.
  \end{align}
  Hence the tangential limit system \eqref{E:Eu_Limit_Tangential} becomes
  \begin{align*}
    \partial_tv+\overline{\nabla}_vv+\nabla_\Gamma q = 0, \quad \mathrm{div}_\Gamma v = 0 \quad\text{on}\quad \Gamma\times(0,T),
  \end{align*}
  which is the same as the Euler system on a manifold derived by Arnol$'$d~\cite{A66,A89bk} (see also Ebin and Marsden~\cite{EM70}).
  Also, applying $v\cdot\nu=0$, \eqref{E:Gauss_Formula}, and the fact that $\nu$ is independent of time to \eqref{E:Eu_Limit_Normal_Move} we obtain
  \begin{align} \label{E:Eu_Normal_Limit_Fixed}
    q^1 = - \partial^\bullet_vv\cdot\nu = -\partial_t(v\cdot\nu)-\{(v\cdot\nabla)v\}\cdot\nu = -Av\cdot v,
  \end{align}
  which does not vanish in general even if the surface is stationary.
\end{remark}

\begin{remark} \label{R:Incompressibility}
  For the existence of a surface incompressible fluid obeying \eqref{E:Eu_Continuity_Limit} it is required that the area of the moving surface $\Gamma(t)$ is preserved in time.
  Indeed, by the Leibniz formula (see~\cite[Lemma~2.2]{DE07}) with a velocity field $v$ on $S_T$ satisfying $v\cdot\nu=V_\Gamma^N$ and \eqref{E:Eu_Continuity_Limit} we have
  \begin{align} \label{E:Area_Preserve}
    \frac{d}{dt}|\Gamma(t)| = \frac{d}{dt}\int_{\Gamma(t)}1\,d\mathcal{H}^2 = \int_{\Gamma(t)}\mathrm{div}_\Gamma v\,d\mathcal{H}^2 = 0,
  \end{align}
  where $|\Gamma(t)|$ is the area of $\Gamma(t)$.
  Similarly, when the moving thin domain $\Omega_\varepsilon(t)$ is filled with an incompressible fluid satisfying \eqref{E:Eu_Continuity} and the impermeable boundary condition \eqref{E:Eu_Boundary}, its volume $|\Omega_\varepsilon(t)|$ must remain constant by the Reynolds transport theorem (see e.g.~\cite{G81bk}):
  \begin{align*}
    \frac{d}{dt}|\Omega_\varepsilon(t)| &= \frac{d}{dt}\int_{\Omega_\varepsilon(t)}1\,dx = \int_{\Omega_\varepsilon(t)}V_\varepsilon^N\,d\mathcal{H}^2 \\
    &= \int_{\partial\Omega_\varepsilon(t)}u\cdot\nu_\varepsilon\,d\mathcal{H}^2 = \int_{\Omega_\varepsilon(t)}\mathrm{div}\,u\,dx = 0.
  \end{align*}
  By the change of variables formula \eqref{E:Co-area_NBD} the volume of $\Omega_\varepsilon(t)$ is expressed as
  \begin{align*}
    |\Omega_\varepsilon(t)| &= \int_{\Omega_\varepsilon(t)}1\,dx = \int_{\Gamma(t)}\int_{-\varepsilon}^\varepsilon J(y,t,\rho)\,d\rho\,d\mathcal{H}^2 \\
    &= 2\varepsilon|\Gamma(t)|+\frac{2}{3}\varepsilon^3\int_{\Gamma(t)}K\,d\mathcal{H}^2.
  \end{align*}
  Hence we need to assume
  \begin{align*}
    \frac{d}{dt}|\Gamma(t)| = 0, \quad \frac{d}{dt}\int_{\Gamma(t)}K\,d\mathcal{H}^2 = 0
  \end{align*}
  for the existence of an incompressible fluid in the $\varepsilon$-tubular neighborhood $\Omega_\varepsilon(t)$ of $\Gamma(t)$ for all $\varepsilon>0$.
  However, by the Gauss-Bonnet theorem we have
  \begin{align*}
    \int_{\Gamma(t)}K\,d\mathcal{H}^2 = 2\pi\chi(\Gamma(t)),
  \end{align*}
  where $\chi(\Gamma(t))$ is the Euler characteristic of $\Gamma(t)$ (see e.g.~\cite[Section~C.5]{T11bk}).
  Since the Euler characteristic is a topological invariant and the moving surface $\Gamma(t)$ does not change its topology, the integral of the Gaussian curvature $K$ over $\Gamma(t)$ is constant in time.
  Therefore, only the area preserving condition \eqref{E:Area_Preserve} on $\Gamma(t)$ is necessary for the existence of incompressible fluids on $\Gamma(t)$ and in $\Omega_\varepsilon(t)$ for all $\varepsilon>0$.
  Note that this assertion is valid only for a moving surface in $\mathbb{R}^3$ or a moving hypersurface in $\mathbb{R}^4$.
  Indeed, when $\Gamma(t)$ is a moving hypersurface in $\mathbb{R}^n$ with $n>4$, the Jacobian $J(y,t,\rho)$ is a polynomial in $\rho$ of degree greater than three (see e.g.~\cite[Section~14.6]{GT01bk} and~\cite[Section~5.1]{M17}) and thus we need more constraints on the motion of $\Gamma(t)$.
\end{remark}

\begin{proof}[Proof of Theorem~\ref{T:Eu_Limit}]
  For the sake of simplicity, we use the abbreviations
  \begin{align} \label{E:Abbrev}
    f(\pi,t) = f(\pi(x,t),t), \quad R(d^k) = R(d(x,t)^k)
  \end{align}
  for a function $f$ on $S_T$ and $k\in\mathbb{N}$.
  Since $\nu_\varepsilon$ and $V_\varepsilon^N$ are given by \eqref{E:N_Vect_Dom} and \eqref{E:N_Vel_Dom}, the boundary condition \eqref{E:Eu_Boundary} reads
  \begin{align*}
    u(x,t)\cdot\nu(\pi,t) = V_\Gamma^N(\pi,t), \quad x\in\partial\Omega_\varepsilon(t).
  \end{align*}
  We substitute \eqref{E:Eu_U_Expand} for $u$ in the above equality.
  Then
  \begin{align*}
    v(\pi,t)\cdot\nu(\pi,t)\pm\varepsilon v^1(\pi,t)\cdot\nu(\pi,t)+O(\varepsilon^2) = V_\Gamma^N(\pi,t)
  \end{align*}
  when $d(x,t)=\pm\varepsilon$ (double-sign corresponds).
  Since $v(\pi,t)$, $v^1(\pi,t)$, $\nu(\pi,t)$, and $V_\Gamma^N(\pi,t)$ are independent of $\varepsilon$, it follows from the above equation that
  \begin{align}
    v(\pi,t)\cdot\nu(\pi,t) &= V_\Gamma^N(\pi,t), \label{E:Eu_V0_N} \\
    v^1(\pi,t)\cdot\nu(\pi,t) &= 0. \label{E:Eu_V1_N}
  \end{align}
  The first statement of the theorem follows from the equality \eqref{E:Eu_V0_N}.
  Let us write $v=V_\Gamma^N\nu+v^T$ with a tangential velocity field $v^T$ on $\Gamma(t)$ and derive the equations \eqref{E:Eu_Momentum_Limit} and \eqref{E:Eu_Continuity_Limit}.
  By \eqref{E:Nabla_Dist} and \eqref{E:Nabla_Composite} we have
  \begin{align} \label{E:Eu_Nabla_U}
    \nabla u(x,t) &= \nabla\bigl(v(\pi,t)\bigr)+\nabla d(x,t)\otimes v^1(\pi,t)+R(d) \\
    &= \nabla_\Gamma v(\pi,t)+\nu(\pi,t)\otimes v^1(\pi,t)+R(d) \notag
  \end{align}
  and
  \begin{align} \label{E:Nabla_P}
    \nabla p(x,t) = \nabla_\Gamma q(\pi,t)+q^1(\pi,t)\nu(\pi,t)+R(d).
  \end{align}
  Also, by \eqref{E:Dt_Dist_NBD} and \eqref{E:Dt_Composite},
  \begin{align} \label{E:Eu_Dt_U}
    \partial_tu(x,t) &= \partial_t\bigl(v(\pi,t)\bigr)+\partial_td(x,t)v^1(\pi,t)+R(d) \\
    &= \partial^\circ v(\pi,t)-V_\Gamma^N(\pi,t)v^1(\pi,t)+R(d). \notag
  \end{align}
  From \eqref{E:Eu_Nabla_U} the gradient of the $j$-th component of $u$ is
  \begin{align*}
    \nabla u_j(x,t) = \nabla_\Gamma v_j(\pi,t)+v_j^1(\pi,t)\nu(\pi,t)+R(d).
  \end{align*}
  We take the inner product of this equation and \eqref{E:Eu_U_Expand}, and then apply \eqref{E:Eu_V0_N} and $v\cdot\nabla_\Gamma v_j=v^T\cdot\nabla_\Gamma v_j$ to get the $j$-th component of the inertia term
  \begin{align*}
    u(x,t)\cdot\nabla u_j(x,t) = v^T(\pi,t)\cdot\nabla_\Gamma v_j(\pi,t)+V_\Gamma^N(\pi,t)v_j^1(\pi,t)+R(d).
  \end{align*}
  Hence the inertia term $(u\cdot\nabla)u$ is of the form
  \begin{align}\label{E:Eu_Inertia}
    [(u\cdot\nabla)u](x,t) = [(v^T\cdot\nabla_\Gamma)v](\pi,t)+V_\Gamma^N(\pi,t)v^1(\pi,t)+R(d).
  \end{align}
  Substituting \eqref{E:Nabla_P}, \eqref{E:Eu_Dt_U}, and \eqref{E:Eu_Inertia} for \eqref{E:Eu_Momentum} and applying \eqref{E:Material_and_Normal} we obtain
  \begin{align*}
    \partial^\bullet_vv(\pi,t)+\nabla_\Gamma q(\pi,t)+q^1(\pi,t)\nu(\pi,t) = R(d).
  \end{align*}
  In this equation, each term on the left-hand side is independent of $d$.
  Therefore, the equation \eqref{E:Eu_Momentum_Limit} should be satisfied.

  Finally, by \eqref{E:Eu_V1_N} and \eqref{E:Eu_Nabla_U} we have
  \begin{align*}
    \mathrm{div}\,u(x,t) &= \mathrm{tr}[\nabla u(x,t)] = \mathrm{div}_\Gamma v(\pi,t)+\nu(\pi,t)\cdot v^1(\pi,t)+R(d) \\
    &= \mathrm{div}_\Gamma v(\pi,t)+R(d)
  \end{align*}
  and thus the equation \eqref{E:Eu_Continuity} reads $\mathrm{div}_\Gamma v(\pi,t) = R(d)$.
  Since the left-hand side is independent of $d$, we conclude that $v$ satisfies the equation \eqref{E:Eu_Continuity_Limit}.
\end{proof}

\section{Limit equations of the Navier-Stokes equations} \label{S:Limit_NS}
In this section, we consider the incompressible Navier-Stokes equations in $\Omega_\varepsilon(t)$:
\begin{alignat}{2}
  \partial_tu+(u\cdot\nabla)u+\nabla p &= \mu_0\Delta u &\quad &\text{in}\quad Q_{\varepsilon,T}, \label{E:NS_Momentum}\\
  \mathrm{div}\,u &= 0 &\quad &\text{in}\quad Q_{\varepsilon,T}. \label{E:NS_Continuity}
\end{alignat}
Here $u=(u_1,u_2,u_3)$ is the velocity of a bulk fluid, $p$ is the pressure, and $\mu_0>0$ is the viscosity coefficient.
On these equations we impose the (perfect slip) Navier boundary condition of the form
\begin{alignat}{2}
  u\cdot\nu_\varepsilon &= V_\varepsilon^N &\quad &\text{on}\quad \partial_\ell Q_{\varepsilon,T}, \label{E:NS_B_Norm} \\
  [D(u)\nu_\varepsilon]_{\text{tan}} &= 0 &\quad &\text{on}\quad \partial_\ell Q_{\varepsilon,T}. \label{E:NS_B_Tan}
\end{alignat}
Here $[a]_{\text{tan}}$ denotes the tangential component to $\partial\Omega_\varepsilon(t)$ of a vector $a\in\mathbb{R}^3$ and $D(u)$ is the strain rate tensor given by
\begin{align*}
  D(u) := \frac{\nabla u+(\nabla u)^T}{2},
\end{align*}
where $(\nabla u)^T$ is the transposed matrix of $\nabla u$.

In order to derive limit equations of the Navier-Stokes equations \eqref{E:NS_Momentum}--\eqref{E:NS_B_Tan} we expand the velocity field $u$ with respect to the signed distance $d(x,t)$ as
\begin{multline} \label{E:NS_U_Expand}
  u(x,t) = v(\pi(x,t),t)+d(x,t)v^1(\pi(x,t),t) \\
  +d(x,t)^2v^2(\pi(x,t),t)+R(d(x,t)^3)
\end{multline}
and the pressure $p$ as \eqref{E:P_Expand}.
We need to expand $u$ up to the second order term in $d(x,t)$ since the momentum equation \eqref{E:NS_Momentum} has the second order derivatives of $u$.

\begin{theorem} \label{T:NS_Limit}
  Let $u$ and $p$ satisfy the Navier-Stokes equations \eqref{E:NS_Momentum}--\eqref{E:NS_B_Tan} in the moving thin domain $\Omega_\varepsilon(t)$.
  Then the normal component of the zeroth order term $v$ in the expansion \eqref{E:NS_U_Expand} is equal to the outward normal velocity of the moving surface $\Gamma(t)$, i.e. $v\cdot\nu=V_\Gamma^N$.
  Moreover, the velocity field $v$ and the zeroth and first order terms $q$ and $q^1$ in the expansion \eqref{E:P_Expand} satisfy
  \begin{align}
    \partial^\bullet_vv+\nabla_\Gamma q+q^1\nu = 2\mu_0\mathrm{div}_\Gamma(P_\Gamma D^{tan}(v)P_\Gamma) \quad&\text{on}\quad S_T, \label{E:NS_Momentum_Limit} \\
    \mathrm{div}_\Gamma v = 0 \quad&\text{on}\quad S_T. \label{E:NS_Continuity_Limit}
  \end{align}
  Here $D^{tan}(v)$ is the tangential strain rate given by \eqref{E:Tan_Strain}.
\end{theorem}

\begin{remark} \label{R:NS_Limit_Comparison}
  As in Remark~\ref{R:Eu_Limit_Comparison}, the normal component of $v$ is equal to $V_\Gamma^N$, the tangential component of $v$ and the scalar function $q$ are determined by
  \begin{align} \label{E:NS_Limit_Tan_move}
    P_\Gamma\partial^\bullet_vv+\nabla_\Gamma q = 2\mu_0P_\Gamma\mathrm{div}_\Gamma(P_\Gamma D^{tan}(v)P_\Gamma), \quad \mathrm{div}_\Gamma v = 0 \quad\text{on}\quad S_T,
  \end{align}
  and the scalar function $q^1$ is given by the normal component of \eqref{E:NS_Momentum_Limit}.
  The tangential system \eqref{E:NS_Limit_Tan_move} is the same as the tangential incompressible Navier-Stokes-Scriven-Koba (NSSK) system in~\cite{KLG16} with constant density (see (4.4) in~\cite{KLG16}).

  When $\Gamma(t)=\Gamma$ is fixed in time, the tangential system \eqref{E:NS_Limit_Tan_move} is the same as the incompressible Navier-Stokes system on a fixed manifold derived by Taylor~\cite{T92}
  \begin{align} \label{E:NS_Manifold}
    \partial_tv+\overline{\nabla}_vv+\nabla_\Gamma q = \mu_0(\Delta_Bv+Kv), \quad \mathrm{div}_\Gamma v = 0 \quad\text{on}\quad \Gamma\times(0,T)
  \end{align}
  for a tangential velocity field $v$ on $\Gamma$, although the authors of~\cite{KLG16} claim that the system \eqref{E:NS_Limit_Tan_move} on the stationary surface $\Gamma$ is different from Taylor's model \eqref{E:NS_Manifold} (see Remark~\ref{R:Deformation_Tensor} below).
  Indeed, when the surface $\Gamma$ is stationary, i.e. $V_\Gamma^N=0$, the velocity field $v$ in the system \eqref{E:NS_Limit_Tan_move} is tangential and by applying \eqref{E:Projection_Material_Stationary} to the left-hand side of the first equation in \eqref{E:NS_Limit_Tan_move} we obtain
  \begin{align*}
    \partial_tv+\overline{\nabla}_vv+\nabla_\Gamma q = 2\mu_0P_\Gamma\mathrm{div}_\Gamma(P_\Gamma D^{tan}(v)P_\Gamma), \quad \mathrm{div}_\Gamma v = 0 \quad\text{on}\quad \Gamma\times(0,T).
  \end{align*}
  Moreover, since $v$ is tangential and satisfies $\mathrm{div}_\Gamma v=0$, the right-hand side of the first equation in the above system is the same as that in Taylor's system \eqref{E:NS_Manifold} by \eqref{E:Compare_Viscous}.
  Hence the tangential incompressible Navier-Stokes system \eqref{E:NS_Limit_Tan_move} on the stationary surface $\Gamma$ agrees with the system \eqref{E:NS_Manifold} given by Taylor.

  As in the case of the Euler equations (see Remark~\ref{R:Eu_Limit_Comparison}), when the surface is stationary the function $q^1$ in \eqref{E:NS_Momentum_Limit} is given by
  \begin{align*}
    q^1 = \{-\partial^\bullet_vv+2\mu_0\mathrm{div}_\Gamma(P_\Gamma D^{tan}(v)P_\Gamma)\}\cdot\nu = -Av\cdot v+2\mu_0\mathrm{tr}[A\nabla_\Gamma v],
  \end{align*}
  where the second equality follows from \eqref{E:Viscous_General} and \eqref{E:Eu_Normal_Limit_Fixed}.
  From this formula we observe that $q^1$ does not vanish in general even if the surface is stationary.
\end{remark}

\begin{remark} \label{R:Deformation_Tensor}
  The authors of \cite{KLG16} argue that the tangential incompressible Navier-Stokes system \eqref{E:NS_Limit_Tan_move} on a stationary surface $\Gamma$ is different from the Navier-Stokes system \eqref{E:NS_Manifold} on a manifold given by Taylor~\cite{T92}, which is inconsistent with our argument in Remark~\ref{R:NS_Limit_Comparison}.
  Unfortunately, there seems to be a flaw in derivation of Taylor's system \eqref{E:NS_Manifold} in~\cite[Section~5]{KLG16}.
  The authors of~\cite{KLG16} applied an energetic variational approach with the dissipation energy given by the tangential strain rate $D^{tan}(v)=\{\nabla_\Gamma v+(\nabla_\Gamma v)^T\}/2$ to obtain \eqref{E:NS_Manifold}.
  In their derivation of \eqref{E:NS_Manifold} they claim that $P_\Gamma\mathrm{div}_\Gamma\bigl(P_\Gamma D^{tan}(v)\bigr)=\Delta_B v+Kv$ holds on $\Gamma$ when $\Gamma$ is stationary and $v$ is tangential and satisfies $\mathrm{div}_\Gamma v=0$ (see the argument after~\cite[Theorem~5.1]{KLG16}).
  However, we have
  \begin{align*}
    2P_\Gamma\mathrm{div}_\Gamma\bigl(P_\Gamma D^{tan}(v)\bigr) = \Delta_B v+Kv-A^2v
  \end{align*}
  for any tangential vector field $v$ on $\Gamma$ satisfying $\mathrm{div}_\Gamma v=0$, since the sum of the first two terms on the right-hand side is equal to $2P_\Gamma\mathrm{div}_\Gamma(P_\Gamma D^{tan}(v)P_\Gamma)$ by \eqref{E:Compare_Viscous} and
  \begin{multline*}
    2P_\Gamma\mathrm{div}_\Gamma\bigl(P_\Gamma D^{tan}(v)\bigr)-2P_\Gamma\mathrm{div}_\Gamma(P_\Gamma D^{tan}(v)P_\Gamma) \\
    = 2P_\Gamma\mathrm{div}_\Gamma\bigl(P_\Gamma D^{tan}(v)(\nu\otimes\nu)\bigr) = -A^2v
  \end{multline*}
  holds by the same calculations as in the proof of Lemma~\ref{L:Viscous_General} (see Appendix~\ref{S:Appendix_A}).

  It seems that their choice of the dissipation energy for derivation of \eqref{E:NS_Manifold} comes from a subtle misunderstanding of the strain rate tensor in Taylor's model, which is called the deformation tensor in~~\cite{MT01,T92}.
  Taylor~\cite{T92} defined the deformation tensor $\mathrm{Def}\,v$ for a tangential vector field $v$ on $\Gamma$ as a symmetric tensor field of type $(0,2)$ on the manifold $\Gamma$ (see e.g.~\cite[Chapter~12]{L13bk} for tensor fields) satisfying
  \begin{align} \label{E:Def_Deformation}
    (\mathrm{Def}\,v)(X,Y) = \frac{1}{2}\Bigl(\overline{\nabla}_Xv\cdot Y+X\cdot\overline{\nabla}_Yv\Bigr), \quad X,Y\in C^\infty(T\Gamma),
  \end{align}
  where $C^\infty(T\Gamma)$ is the set of all smooth tangential vector fields on $\Gamma$.
  (See also (2.3) in~\cite{MT01}. Note that (2.3) in~\cite{MT01} is a formula for one-forms on $\Gamma$ and here we identify tangential vector fields on $\Gamma$ with one-forms on $\Gamma$ via raising and lowering indices.)
  Let us show that the right-hand side of \eqref{E:Def_Deformation} is equal to $\{D^{tan}(v)X\}\cdot Y$.
  By the Gauss formula \eqref{E:Gauss_Formula_Another} and the fact that the covariant derivative $\overline{\nabla}_Xv$ is tangential,
  \begin{align*}
    \overline{\nabla}_Xv = P_\Gamma\{(X\cdot\nabla_\Gamma)v\} = P_\Gamma(\nabla_\Gamma v)^TX \quad\text{on}\quad \Gamma,
  \end{align*}
  where the second equality just follows from our notation on the tangential gradient matrix (see Section~\ref{S:Preliminaries}).
  From this formula and the fact that $P_\Gamma$ is symmetric and that $Y$ is tangential it follows that
  \begin{align*}
    \overline{\nabla}_Xv\cdot Y = \{P_\Gamma(\nabla_\Gamma v)^TX\}\cdot Y = \{(\nabla_\Gamma v)^TX\}\cdot(P_\Gamma Y) = \{(\nabla_\Gamma v)^TX\}\cdot Y.
  \end{align*}
  Similarly we have $X\cdot\overline{\nabla}_Yv = X\cdot\{(\nabla_\Gamma v)^TY\} = \{(\nabla_\Gamma v)X\}\cdot Y$ and thus
  \begin{align*}
    \frac{1}{2}\Bigl(\overline{\nabla}_Xv\cdot Y+X\cdot\overline{\nabla}_Yv\Bigr) = \frac{1}{2}\Bigl(\{\nabla_\Gamma v+(\nabla_\Gamma v)^T\}X\Bigr)\cdot Y = \{D^{tan}(v)X\}\cdot Y.
  \end{align*}
  Therefore, for any $X$, $Y\in C^\infty(T\Gamma)$ the equality
  \begin{align} \label{E:Wrong_Deformation}
    (\mathrm{Def}\,v)(X,Y) = \{D^{tan}(v)X\}\cdot Y
  \end{align}
  holds.
  Therefore, the deformation tensor $\mathrm{Def}\,v$ can be identified with the restriction on $C^\infty(T\Gamma)\times C^\infty(T\Gamma)$ of the symmetric bilinear map
  \begin{align*}
    \mathcal{T}_{D^{tan}(v)}\colon C^\infty(\Gamma)^3\times C^\infty(\Gamma)^3 \to C^\infty(\Gamma), \quad (F,G) \mapsto \{D^{tan}(v)F\}\cdot G.
  \end{align*}
  Here $C^\infty(\Gamma)$ denotes the set of all smooth functions on $\Gamma$ and $C^\infty(\Gamma)^3$ is the set of all smooth three-dimensional vector fields on $\Gamma$ not necessarily tangential.
  However, it does not mean that $\mathrm{Def}\,v$ can be identified with the matrix $D^{tan}(v)$.
  Since $\mathrm{Def}\,v$ is a tensor field of type $(0,2)$ on the manifold $\Gamma$, for any $X\in C^\infty(T\Gamma)$ the mapping
  \begin{align*}
    (\mathrm{Def}\,v)(X,\cdot)\colon C^\infty(T\Gamma)\to C^\infty(\Gamma), \quad Y \mapsto (\mathrm{Def}\,v)(X,Y)
  \end{align*}
  is a linear map from $C^\infty(T\Gamma)$ into $C^\infty(\Gamma)$, i.e. a one-form on $\Gamma$.
  By identifying one-forms on $\Gamma$ with tangential vector fields on $\Gamma$ via raising and lowering indices, we may consider $(\mathrm{Def}\,v)(X,\cdot)=(\mathrm{Def}\,v)X$ as a tangential vector field on $\Gamma$.
  On the other hand, for a tangential vector field $X$ on $\Gamma$ the vector field $D^{tan}(v)X$ is not tangential in general, even if $v$ is tangential to $\Gamma$.
  Indeed, since $(\nabla_\Gamma v)^T\nu=(\nabla_\Gamma v)^TP_\Gamma\nu=0$ and $(\nabla_\Gamma v)\nu=-(\nabla_\Gamma\nu)v=Av$, where the second relation follows from the fact that $v$ is tangential, we have
  \begin{align*}
    D^{tan}(v)\nu = \frac{1}{2}\{(\nabla_\Gamma v)\nu+(\nabla_\Gamma v)^T\nu\} = \frac{1}{2}Av.
  \end{align*}
  From this equality and the symmetry of the matrix $D^{tan}(v)$ it follows that
  \begin{align*}
    D^{tan}(v)X\cdot\nu = X\cdot D^{tan}(v)\nu = \frac{1}{2}X\cdot Av
  \end{align*}
  for any tangential vector field $X$ on $\Gamma$.
  The last term does not vanish and thus the vector field $D^{tan}(v)X$ is not tangential on $\Gamma$ in general.

  To give a proper interpretation of the deformation tensor as a matrix, we observe that in \eqref{E:Wrong_Deformation} the vector fields $X$ and $Y$ are tangential to $\Gamma$ and thus
  \begin{align*}
    \{D^{tan}(v)X\}\cdot Y = \{D^{tan}(v)P_\Gamma X\}\cdot(P_\Gamma Y) = \{P_\Gamma D^{tan}(v)P_\Gamma X\}\cdot Y
  \end{align*}
  by the symmetry of the orthogonal projection $P_\Gamma$.
  Then \eqref{E:Wrong_Deformation} becomes
  \begin{align*}
    (\mathrm{Def}\,v)(X,Y) = \{P_\Gamma D^{tan}(v)P_\Gamma X\}\cdot Y
  \end{align*}
  for all tangential vector fields $X$ and $Y$ on $\Gamma$.
  Moreover, the matrix $P_\Gamma D^{tan}(v)P_\Gamma$ is symmetric and for any $X\in C^\infty(T\Gamma)$ the vector field $P_\Gamma D^{tan}(v)P_\Gamma X$ is tangential to $\Gamma$.
  Therefore, we may identify the deformation tensor
  \begin{align*}
    \mathrm{Def}\,v=\mathcal{T}_{D^{tan}(v)}|_{C^\infty(T\Gamma)\times C^\infty(T\Gamma)}\colon C^\infty(T\Gamma)\times C^\infty(T\Gamma) \to C^\infty(\Gamma)
  \end{align*}
  with the symmetric matrix $P_\Gamma D^{tan}(v)P_\Gamma$.

  The matrix $P_\Gamma D^{tan}(v)P_\Gamma$ is called a projected strain rate in~\cite{KLG16} and employed to define the dissipation energy in their energetic variational method for derivation of the incompressible NSSK system on the moving surface (see~\cite[Lemma~3.4 and Section~4]{KLG16}).
  Therefore, the strain rate tensor in Taylor's system~\eqref{E:NS_Manifold} is the same as that in the tangential incompressible Navier-Stokes system~\eqref{E:NS_Limit_Tan_move}.
\end{remark}

\begin{proof}[Proof of Theorem~\ref{T:NS_Limit}]
  As in the proof of Theorem~\ref{T:Eu_Limit} we use the abbreviations \eqref{E:Abbrev}.
  Due to the first boundary condition \eqref{E:NS_B_Norm} we have
  \begin{align}
    v(\pi,t)\cdot\nu(\pi,t) &= V_\Gamma^N(\pi,t), \label{E:NS_V0_N} \\
    v^1(\pi,t)\cdot\nu(\pi,t) &= 0, \label{E:NS_V1_N} \\
    v^2(\pi,t)\cdot\nu(\pi,t) &= 0 \label{E:NS_V2_N}
  \end{align}
  and the surface divergence-free condition \eqref{E:NS_Continuity_Limit} for $v$ by the same argument as in the proof of Theorem~\ref{T:Eu_Limit}.
  Moreover, we already calculated the expansion of the left-hand side of \eqref{E:NS_Momentum} in the proof of Theorem~\ref{T:Eu_Limit}:
  \begin{multline} \label{E:NS_Left_Exp}
    \partial_tu(x,t)+[(u\cdot\nabla)u](x,t)+\nabla p(x,t) \\
    = \partial^\bullet_vv(\pi,t)+\nabla_\Gamma q(\pi,t)+q^1(\pi,t)\nu(\pi,t)+R(d).
  \end{multline}
  Let us show that the expansion of the viscous term $\Delta u$ is of the form
  \begin{align} \label{E:NS_Visc_Exp}
    \Delta u(x,t) = 2[\mathrm{div}_\Gamma(P_\Gamma D^{tan}(v)P_\Gamma)](\pi,t)+R(d).
  \end{align}
  Since $\Delta u=2\mathrm{div}\,D(u)$ holds by the divergence-free condition \eqref{E:NS_Continuity}, we consider the expansion in powers of $d$ of the strain rate tensor $D(u)$.
  We differentiate both sides of \eqref{E:NS_U_Expand} with respect to $x$ and apply \eqref{E:Nabla_Dist} and \eqref{E:Nabla_Composite} to get
  \begin{multline} \label{E:NS_Nabla_U}
    \nabla u(x,t) = \nabla_\Gamma v(\pi,t)+[\nu\otimes v^1](\pi,t) \\
    +d(x,t)\{[A\nabla_\Gamma v](\pi,t)+\nabla_\Gamma v^1(\pi,t)+2[\nu\otimes v^2](\pi,t)\}+R(d^2).
  \end{multline}
  Hence the strain rate tensor of $u$ is expressed as
  \begin{align} \label{E:NS_Strain_Exp}
    D(u)(x,t) = S(\pi,t)+d(x,t)S^1(\pi,t)+R(d^2),
  \end{align}
  where
  \begin{align}
    S &:= D^{tan}(v)+\frac{\nu\otimes v^1+v^1\otimes\nu}{2}, \label{E:NS_Strain0} \\
    S^1 &:= \frac{A\nabla_\Gamma v+(A\nabla_\Gamma v)^T}{2}+D^{tan}(v^1)+\nu\otimes v^2+v^2\otimes\nu. \label{E:NS_Strain1}
  \end{align}
  Let us write the second boundary condition \eqref{E:NS_B_Tan} in terms of $S$ and $S^1$.
  By \eqref{E:N_Vect_Dom} and \eqref{E:N_Vel_Dom} the boundary condition \eqref{E:NS_B_Tan} reads
  \begin{align*}
    P_\Gamma(\pi,t)D(u)(x,t)\nu(\pi,t) = 0, \quad x\in\partial\Omega_\varepsilon(t).
  \end{align*}
  We substitute \eqref{E:NS_Strain_Exp} for the above $D(u)(x,t)$ to obtain
  \begin{align*}
    P_\Gamma(\pi,t)S(\pi,t)\nu(\pi,t)\pm\varepsilon P_\Gamma(\pi,t)S^1(\pi,t)\nu(\pi,t)+O(\varepsilon^2) = 0
  \end{align*}
  according to $d(x,t)=\pm\varepsilon$ (double-sign corresponds).
  Since the matrices $S(\pi,t)$, $S^1(\pi,t)$, $P_\Gamma(\pi,t)$, and the vector $\nu(\pi,t)$ are independent of $\varepsilon$, we have
  \begin{align}
    P_\Gamma(\pi,t)S(\pi,t)\nu(\pi,t) &= 0  \label{E:NS_Strain0_N} \\
    P_\Gamma(\pi,t)S^1(\pi,t)\nu(\pi,t) &= 0. \label{E:NS_Strain1_N}
  \end{align}
  Substituting \eqref{E:NS_Strain0} for $S$ in \eqref{E:NS_Strain0_N} and observing
  \begin{align*}
    (\nu\otimes v^1)\nu = (v^1\cdot\nu)\nu = 0,\quad (v^1\otimes\nu)\nu = (\nu\cdot\nu)v^1 = v^1, \quad P_\Gamma v^1 = v^1
  \end{align*}
  by \eqref{E:NS_V1_N} we get
  \begin{align} \label{E:NS_V1_Full}
    v^1(\pi,t) = -2P_\Gamma(\pi,t)D^{tan}(v)(\pi,t)\nu(\pi,t).
  \end{align}
  Moreover, we multiply $\nu$ by $S^1$ given by \eqref{E:NS_Strain1} and apply
  \begin{align*}
    (A\nabla_\Gamma v)^T\nu = (\nabla_\Gamma v)^TA\nu = 0, \quad (\nabla_\Gamma v^1)^T\nu = (\nabla_\Gamma v^1)^TP_\Gamma\nu = 0
  \end{align*}
  by the symmetry of $A$ and $P_\Gamma$, $\nabla_\Gamma=P_\Gamma\nabla_\Gamma$, and \eqref{E:W_Normal}, and then use $(\nu\otimes v^2)\nu = 0$ and $(v^2\otimes\nu)\nu = v^2$ by \eqref{E:NS_V2_N} to obtain
  \begin{align} \label{E:NS_S1nu_Express}
    S^1\nu = \frac{1}{2}(A\nabla_\Gamma v+\nabla_\Gamma v^1)\nu+v^2.
  \end{align}
  It is tangential to $\Gamma(t)$ by $\nabla_\Gamma=P_\Gamma\nabla$, \eqref{E:W_Normal} and \eqref{E:NS_V2_N}.
  Hence \eqref{E:NS_Strain1_N} yields
  \begin{align}
    S^1(\pi,t)\nu(\pi,t) = 0. \label{E:NS_S1nu_Full}
  \end{align}
  Now we apply the formula \eqref{E:Div_Matrix} to the expansion \eqref{E:NS_Strain_Exp}.
  Then by the symmetry of $S^1$ (see \eqref{E:NS_Strain1}) and the equality \eqref{E:NS_S1nu_Full} we get
  \begin{align} \label{E:NS_DIV_Strain}
    \mathrm{div}\,D(u)(x,t) = \mathrm{div}_\Gamma S(\pi,t)+R(d).
  \end{align}
  Let us write $S$ in terms of $v$.
  Substituting \eqref{E:NS_V1_Full} for \eqref{E:NS_Strain0}, using the formulas
  \begin{align*}
    (Ma)\otimes b = M(a\otimes b), \quad a\otimes(Mb) = (a\otimes b)M^T
  \end{align*}
  for a square matrix $M$ of order three and three-dimensional vectors $a$ and $b$, and observing $\bigl(P_\Gamma D^{tan}(v)\bigr)^T=D^{tan}(v)P_\Gamma$ by the symmetry of $P_\Gamma$ and $D^{tan}(v)$, we have
  \begin{align*}
    S &= D^{tan}(v)-(\nu\otimes\nu)D^{tan}(v)P_\Gamma-P_\Gamma D^{tan}(v)(\nu\otimes\nu) \\
    &= P_\Gamma D^{tan}(v)P_\Gamma+(\nu\otimes\nu)D^{tan}(v)(\nu\otimes\nu).
  \end{align*}
  Here the second term on the last line vanishes by $(\nu\otimes\nu)\nabla_\Gamma v = (\nabla_\Gamma v)^T(\nu\otimes\nu) = 0$.
  Hence it follows that
  \begin{align} \label{E:NS_Strain0_Full}
    S(\pi,t) = P_\Gamma(\pi,t)D^{tan}(v)(\pi,t)P_\Gamma(\pi,t)
  \end{align}
  and we obtain \eqref{E:NS_Visc_Exp} by applying \eqref{E:NS_DIV_Strain} and \eqref{E:NS_Strain0_Full} to $\Delta u=2\mathrm{div}\,D(u)$.
  Finally, we substitute \eqref{E:NS_Left_Exp} and \eqref{E:NS_Visc_Exp} for the momentum equation \eqref{E:NS_Momentum} to get
  \begin{multline*}
    \partial^\bullet_vv(\pi,t)+\nabla_\Gamma q(\pi,t)+q^1(\pi,t)\nu(\pi,t)+R(d) \\
    = 2\mu_0[\mathrm{div}_\Gamma(P_\Gamma D^{tan}(v)P_\Gamma)](\pi,t)+R(d).
  \end{multline*}
  Since all terms except for $R(d)$ are independent of $d$, we conclude that the equation \eqref{E:NS_Momentum_Limit} should be satisfied.
\end{proof}

\begin{remark} \label{R:NS_Limit_Partial_Slip}
  We may replace the perfect slip condition \eqref{E:NS_B_Tan} by the partial slip condition
  \begin{align*}
    [D(u)\nu_\varepsilon]_{\text{tan}}+k(u^T-v_\Omega^T) = 0 \quad\text{on}\quad \partial_\ell Q_{\varepsilon,T},
  \end{align*}
  where $u^T=(I_3-\nu_\varepsilon\otimes\nu_\varepsilon)u$, $k>0$ is a constant, and $v_\Omega^T(\cdot,t)$ is a given tangential velocity field on $\partial\Omega_\varepsilon(t)$.
  However, it makes the limit velocity overdetermined.
  Indeed, suppose that $v_\Omega^T$ is given by
  \begin{align*}
    v_\Omega^T(x,t) =
    \begin{cases}
      v_{\text{outer}}(\pi(x,t),t) &\text{if}\quad d(x,t) = \varepsilon, \\
      v_{\text{inner}}(\pi(x,t),t) &\text{if}\quad d(x,t) = -\varepsilon,
    \end{cases}
  \end{align*}
  where $v_{\text{outer}}(\cdot,t)$ and $v_{\text{inner}}(\cdot,t)$ are given tangential velocity fields on $\Gamma(t)$. Then the same calculations as in the proof of Theorem~\ref{T:NS_Limit} yield
  \begin{align*}
    v = V_\Gamma^N\nu+\frac{v_{\text{outer}}+v_{\text{inner}}}{2}.
  \end{align*}
  Hence the limit velocity $v$ is completely determined by given velocities while it should satisfy similar equations to \eqref{E:NS_Momentum_Limit} and \eqref{E:NS_Continuity_Limit}.
\end{remark}

\begin{remark} \label{R:NS_Limit_Viscous_Another}
  In the proof of Theorem~\ref{T:NS_Limit} we obtained the expansion \eqref{E:NS_Visc_Exp} of the viscous term $\Delta u$ by using the expansion of the strain rate tensor $D(u)$.
  Here let us expand $\Delta u$ by direct calculations.
  In what follows, we abbreviate $\pi(x,t)$ and $R(d(x,t))$ to $\pi$ and $R(d)$ for $x\in\Omega_\varepsilon(t)$ and suppress the argument $t$.
  By \eqref{E:NS_Nabla_U} the gradient of the $j$-th component of $u$ ($j=1,2,3$) is
  \begin{align} \label{E:NS_ Nabla_Uj}
    \nabla u_j(x) = \nabla_\Gamma v_j(\pi)+v_j^1(\pi)\nu(\pi)+d(x)F_j(\pi)+R(d^2),
  \end{align}
  where $F_j = A\nabla_\Gamma v_j+\nabla_\Gamma v_j^1+2v_j^2\nu$.
  We differentiate both sides of \eqref{E:NS_ Nabla_Uj} with respect to $x$ and apply \eqref{E:Nabla_Dist}, \eqref{E:Nabla_Composite}, and $\nabla_\Gamma\nu=-A$ to get
  \begin{align*}
    \nabla^2 u_j(x) = \nabla_\Gamma^2v_j(\pi)+[\nabla_\Gamma v_j^1\otimes\nu](\pi)-v_j^1(\pi)A(\pi)+[\nu\otimes F_j](\pi)+R(d).
  \end{align*}
  Taking the trace of both sides and observing $A\nabla_\Gamma v_j\cdot\nu=\nabla_\Gamma v_j^1\cdot\nu=0$ we obtain
  \begin{align*}
    \Delta u_j(x) = \Delta_\Gamma v_j(\pi)-v_j^1(\pi)H(\pi)+2v_j^2(\pi)+R(d)
  \end{align*}
  for each $j=1,2,3$ and thus
  \begin{align*}
    \Delta u_j(x) = \Delta_\Gamma v(\pi)-H(\pi)v^1(\pi)+2v^2(\pi)+R(d).
  \end{align*}
  Let us express $v^1$ and $v^2$ in terms of $v$.
  The first order term $v^1$ is given by \eqref{E:NS_V1_Full}, $\nabla_\Gamma=P_\Gamma\nabla_\Gamma$, and $(\nabla_\Gamma v)^T\nu=(\nabla_\Gamma v)^TP_\Gamma\nu=0$:
  \begin{align*}
    v^1 = -2P_\Gamma D^{tan}(v)\nu = -(\nabla_\Gamma v)\nu.
  \end{align*}
  By \eqref{E:NS_S1nu_Express} and \eqref{E:NS_S1nu_Full} we can represent $v^2$ in terms of $v$ and $v^1$ as
  \begin{align*}
    v^2 = -\frac{1}{2}(A\nabla_\Gamma v+\nabla_\Gamma v^1)\nu.
  \end{align*}
  From this it follows that $v^2=0$ since $v^1=-(\nabla_\Gamma v)\nu$ is tangential and thus
  \begin{align*}
    (\nabla_\Gamma v^1)\nu = \nabla_\Gamma(v^1\cdot\nu)-(\nabla_\Gamma\nu)v^1 = Av^1 = -A(\nabla_\Gamma v)\nu.
  \end{align*}
  Hence we obtain another expansion formula of the viscous term
  \begin{align} \label{E:NS_Visc_Exp_Another}
     \Delta u(x) = \Delta_\Gamma v(\pi)+[H(\nabla_\Gamma v)\nu](\pi)+R(d).
   \end{align}
   Comparing the expansions \eqref{E:NS_Visc_Exp} and \eqref{E:NS_Visc_Exp_Another} we expect that the equality
   \begin{align} \label{E:NS_Visc_Equiv}
     2\mathrm{div}_\Gamma(P_\Gamma D^{tan}(v)P_\Gamma) = \Delta_\Gamma v+H(\nabla_\Gamma v)\nu
   \end{align}
   holds for the limit velocity $v$.
   Let us prove this equality.
   By the formula \eqref{E:Viscous_General} for the left-hand side, the proof of \eqref{E:NS_Visc_Equiv} reduces to showing
  \begin{align} \label{E:NS_Visc_Reduced}
    \nabla_\Gamma(\mathrm{div}_\Gamma v) = 0, \quad 2\mathrm{tr}[A\nabla_\Gamma v] = (\Delta_\Gamma v)\cdot\nu.
  \end{align}
  The first equality follows from the surface divergence-free condition \eqref{E:NS_Continuity_Limit} for the limit velocity $v$.
  To obtain the second equality we need to observe the expansion of the divergence-free condition \eqref{E:NS_Continuity} in powers of the signed distance $d$ up to the first order term.
  Taking the trace of \eqref{E:NS_Nabla_U} and using $v^1\cdot\nu=0$ and $v^2=0$ we have
  \begin{align*}
    \mathrm{div}\,u(x) = \mathrm{div}_\Gamma v(\pi)+d(x)\{\mathrm{tr}[A\nabla_\Gamma v](\pi)+\mathrm{div}_\Gamma v^1(\pi)\}+R(d^2).
  \end{align*}
  Since the left-hand side vanishes for all $x\in\Omega_\varepsilon(t)$ by \eqref{E:NS_Continuity}, observing the first order term in $d(x)$ on the right-hand side we obtain
  \begin{align} \label{E:NS_Visc_Div_First}
    \mathrm{tr}[A\nabla_\Gamma v]+\mathrm{div}_\Gamma v^1 = 0.
  \end{align}
  To the second term on the left-hand side we apply $v^1=-(\nabla_\Gamma v)\nu$.
  Then since
  \begin{align*}
    \mathrm{div}_\Gamma[(\nabla_\Gamma v)\nu] &= (\mathrm{div}_\Gamma\nabla_\Gamma v)\cdot\nu+\mathrm{tr}[(\nabla_\Gamma\nu)^T\nabla_\Gamma v], \\
    &= (\Delta_\Gamma v)\cdot\nu-\mathrm{tr}[A^T\nabla_\Gamma v]
  \end{align*}
  and the Weingarten map $A$ is symmetric, the equality \eqref{E:NS_Visc_Div_First} becomes
  \begin{align*}
    2\mathrm{tr}[A\nabla_\Gamma v]-(\Delta_\Gamma v)\cdot\nu = 0.
  \end{align*}
  Hence the second equality in \eqref{E:NS_Visc_Reduced} holds and \eqref{E:NS_Visc_Equiv} follows.
\end{remark}

\section{Energy identities} \label{S:Energy_Identity}
The purpose of this section is to find a relation between energy identities of the Euler and Navier-Stokes equations in the moving thin domains and those of the limit equations on the moving surface.
We first derive the energy identities from the equations and then show that the energy identities of the limit surface equations are also derived as thin width limits of those of the original bulk equations.

\subsection{Euler equations}

\begin{lemma} \label{L:Energy_Eu_Domain}
  Let $u$ and $p$ satisfy the Euler equations \eqref{E:Eu_Momentum}--\eqref{E:Eu_Boundary} in the moving thin domain $\Omega_\varepsilon(t)$.
  Then we have
  \begin{align} \label{E:Energy_Eu_Domain}
    \frac{d}{dt}\int_{\Omega_\varepsilon(t)}\frac{|u|^2}{2}\,dx = -\int_{\partial\Omega_\varepsilon(t)}pV_\varepsilon^N\,d\mathcal{H}^2.
  \end{align}
\end{lemma}

The identity \eqref{E:Energy_Eu_Domain} means that the rate of change of the kinetic energy of the incompressible perfect fluid in a moving domain is equal to the rate of work done by the pressure caused by the motion of the boundary.

\begin{proof}
  By the Reynolds transport theorem (see~\cite{G81bk}) and \eqref{E:Eu_Momentum},
  \begin{align} \label{Pf1:Energy_Eu_Domain}
    \frac{d}{dt}\int_{\Omega_\varepsilon(t)}\frac{|u|^2}{2}\,dx &= \int_{\Omega_\varepsilon(t)}u\cdot\partial_tu\,dx+\int_{\partial\Omega_\varepsilon(t)}\frac{|u|^2}{2}V_\varepsilon^N\,d\mathcal{H}^2 \\
    &= \int_{\Omega_\varepsilon(t)}u\cdot\{-(u\cdot\nabla)u-\nabla p\}\,dx+\int_{\partial\Omega_\varepsilon(t)}\frac{|u|^2}{2}V_\varepsilon^N\,d\mathcal{H}^2. \notag
  \end{align}
  By integration by parts and the equations \eqref{E:Eu_Continuity} and \eqref{E:Eu_Boundary} we have
  \begin{multline*}
    \int_{\Omega_\varepsilon(t)}u\cdot(u\cdot\nabla)u\,dx \\
    \begin{aligned}
      &= \int_{\partial\Omega_\varepsilon(t)}|u|^2(u\cdot\nu_\varepsilon)\,d\mathcal{H}^2-\int_{\Omega_\varepsilon(t)}\{u\cdot(u\cdot\nabla)u+|u|^2\mathrm{div}\,u\}\,dx \\
      &= \int_{\partial\Omega_\varepsilon(t)}|u|^2V_\varepsilon^N\,d\mathcal{H}^2-\int_{\Omega_\varepsilon(t)}u\cdot(u\cdot\nabla)u\,dx.
    \end{aligned}
  \end{multline*}
  Therefore,
  \begin{align} \label{E:Int_Inartia_Domain}
    \int_{\Omega_\varepsilon(t)}u\cdot(u\cdot\nabla)u\,dx = \int_{\partial\Omega_\varepsilon(t)}\frac{|u|^2}{2}V_\varepsilon^N\,d\mathcal{H}^2.
  \end{align}
  On the other hand, by integration by parts
  \begin{align*}
    \int_{\Omega_\varepsilon(t)}u\cdot\nabla p\,dx = \int_{\partial\Omega_\varepsilon(t)}(u\cdot\nu_\varepsilon)p\,d\mathcal{H}^2-\int_{\Omega_\varepsilon(t)}(\mathrm{div}\,u)p\,dx
  \end{align*}
  and we apply \eqref{E:Eu_Continuity} and \eqref{E:Eu_Boundary} to the right-hand side to get
  \begin{align} \label{E:Int_Pressure_Domain}
    \int_{\Omega_\varepsilon(t)}u\cdot\nabla p\,dx = \int_{\partial\Omega_\varepsilon(t)}pV_\varepsilon^N\,d\mathcal{H}^2.
  \end{align}
  Substituting \eqref{E:Int_Inartia_Domain} and \eqref{E:Int_Pressure_Domain} for \eqref{Pf1:Energy_Eu_Domain} we obtain the energy identity \eqref{E:Energy_Eu_Domain}.
\end{proof}

\begin{lemma} \label{L:Energy_Eu_Surface}
  Let $v$, $q$, and $q^1$ satisfy the limit equations \eqref{E:Eu_Momentum_Limit} and \eqref{E:Eu_Continuity_Limit} of the Euler equations.
  Suppose that the normal component of $v$ is equal to the outward normal velocity of $\Gamma(t)$, i.e. $v\cdot\nu=V_\Gamma^N$.
  Then we have
  \begin{align} \label{E:Energy_Eu_Surface}
    \frac{d}{dt}\int_{\Gamma(t)}\frac{|v|^2}{2}\,d\mathcal{H}^2 = \int_{\Gamma(t)}(qH-q^1)V_\Gamma^N\,d\mathcal{H}^2.
  \end{align}
\end{lemma}

The right-hand side of \eqref{E:Energy_Eu_Surface} represents the rate of work done by the moving surface to the fluid.
Note that it contains the scalar function $q^1$, which corresponds to the normal derivative of the surface pressure.

\begin{proof}
  By the assumption we can write $v=V_\Gamma^N\nu+v^T$ with a tangential velocity field $v^T$ on $\Gamma(t)$.
  We apply the Leibniz formula (see~\cite[Lemma 2.2]{DE07}) with $v=V_\Gamma^N\nu+v^T$ to the integral of $|v|^2/2$ over $\Gamma(t)$.
  (Note that the tangential velocity $v^T$ does not affect the change of the shape of $\Gamma(t)$.)
  Then we have
  \begin{align*}
    \frac{d}{dt}\int_{\Gamma(t)}\frac{|v|^2}{2}\,d\mathcal{H}^2 &= \int_{\Gamma(t)}\left\{\partial_v^\bullet\left(\frac{|v|^2}{2}\right)+\frac{|v|^2}{2}\,\mathrm{div}_\Gamma v\right\}\,d\mathcal{H}^2 \\
    &= \int_{\Gamma(t)}v\cdot\partial^\bullet_vv\,d\mathcal{H}^2+\int_{\Gamma(t)}\frac{|v|^2}{2}\,\mathrm{div}_\Gamma v\,d\mathcal{H}^2.
  \end{align*}
  To the last line we apply the equations \eqref{E:Eu_Momentum_Limit} and \eqref{E:Eu_Continuity_Limit}.
  Then
  \begin{align} \label{Pf1:Energy_Eu_Surface}
    \frac{d}{dt}\int_{\Gamma(t)}\frac{|v|^2}{2}\,d\mathcal{H}^2 = -\int_{\Gamma(t)}v\cdot(\nabla_\Gamma q+q^1\nu)\,d\mathcal{H}^2.
  \end{align}
  For the first term on the right-hand side,
  \begin{align*}
    v\cdot\nabla_\Gamma q = \mathrm{div}_\Gamma(qv)+q\,\mathrm{div}_\Gamma v = -qHV_\Gamma^N+\mathrm{div}_\Gamma(qv^T)
  \end{align*}
  by $v=V_\Gamma^N\nu+v^T$, $\nabla_\Gamma(qV_\Gamma^N)\cdot\nu=0$, $\mathrm{div}_\Gamma\nu=-H$, and \eqref{E:Eu_Continuity_Limit}.
  Moreover, the integral of the surface divergence of the tangential vector field $qv^T$ over $\Gamma(t)$ vanishes by Stokes' theorem since $\Gamma(t)$ is closed.
  Hence we have
  \begin{align} \label{E:Int_Pressure_Tan}
    \int_{\Gamma(t)}v\cdot\nabla_\Gamma q\,d\mathcal{H}^2 = -\int_{\Gamma(t)}qHV_\Gamma^N\,d\mathcal{H}^2.
  \end{align}
  For the second term we have
  \begin{align} \label{E:Int_Pressure_Norm}
    \int_{\Gamma(t)}v\cdot(q^1\nu)\,d\mathcal{H}^2 = \int_{\Gamma(t)}q^1V_\Gamma^N\,d\mathcal{H}^2
  \end{align}
  by $v\cdot\nu=V_\Gamma^N$.
  The energy identity \eqref{E:Energy_Eu_Surface} follows from \eqref{Pf1:Energy_Eu_Surface}, \eqref{E:Int_Pressure_Tan}, and \eqref{E:Int_Pressure_Norm}.
\end{proof}

Let us show that the energy identity \eqref{E:Energy_Eu_Surface} on the moving surface can be derived as a thin width limit of that in the moving thin domain \eqref{E:Energy_Eu_Domain}.
As in Section~\ref{S:Limit_Eu} we expand the velocity $u$ and the pressure $p$ in powers of the signed distance $d$ as \eqref{E:Eu_U_Expand} and \eqref{E:P_Expand} and determine the zeroth order term in $\varepsilon$ of the energy identity \eqref{E:Energy_Eu_Domain}.

\begin{theorem} \label{T:Energy_Eu_Limit}
  Let $u$ and $p$ satisfy the energy identity \eqref{E:Energy_Eu_Domain}.
  Then the zeroth order term $v$ in the expansion \eqref{E:Eu_U_Expand} and the zeroth and first order terms $q$ and $q^1$ in the expansion \eqref{E:P_Expand} satisfy the energy identity \eqref{E:Energy_Eu_Surface}.
\end{theorem}

\begin{proof}
  From the expansion \eqref{E:Eu_U_Expand} we have
  \begin{align*}
    \frac{|u(x,t)|^2}{2} = \frac{|v(\pi(x,t),t)|^2}{2}+d(x,t)V(\pi(x,t),t)+R(d(x,t)^2)
  \end{align*}
  for $x\in\Omega_\varepsilon(t)$, where $V:=v\cdot v^1$.
  Using this expansion we write
  \begin{align*}
    \int_{\Omega_\varepsilon(t)}\frac{|u(x,t)|^2}{2}\,dx = I_1+I_2+I_3,
  \end{align*}
  where
  \begin{align*}
    I_1 &:= \int_{\Omega_\varepsilon(t)}\frac{|v(\pi(x,t),t)|^2}{2}\,dx, \\
    I_2 &:= \int_{\Omega_\varepsilon(t)}d(x,t)V(\pi(x,t),t)\,dx, \\
    I_3 &:= \int_{\Omega_\varepsilon(t)}R(d(x,t)^2)\,dx.
  \end{align*}
  To $I_1$ and $I_2$ we apply the change of variables formula \eqref{E:Co-area_NBD} to get
  \begin{align*}
    I_1 &= \int_{\Gamma(t)}\int_{-\varepsilon}^\varepsilon\frac{|v(y,t)|^2}{2}J(y,t,\rho)\,d\rho\,d\mathcal{H}^2(y) \\
    &= 2\varepsilon\int_{\Gamma(t)}\frac{|v(y,t)|^2}{2}\,d\mathcal{H}^2(y)+\varepsilon^2f_1(\varepsilon,t), \\
    I_2 &= \int_{\Gamma(t)}\int_{-\varepsilon}^\varepsilon\rho V(y,t)J(y,t,\rho)\,d\rho\,d\mathcal{H}^2(y) = \varepsilon^2f_2(\varepsilon,t),
  \end{align*}
  where $f_1$ and $f_2$ are polynomials in $\varepsilon$ with time-dependent coefficients.
  (Note that the Jacobian $J(y,t,\rho)$ given by \eqref{E:Jacobian} is a polynomial in $\rho$ and the principal curvatures of $\Gamma(t)$.)
  Hence
  \begin{align} \label{Pf1:energy_Eu_Limit}
    \frac{dI_1}{dt} = 2\varepsilon\frac{d}{dt}\int_{\Gamma(t)}\frac{|v(y,t)|^2}{2}\,d\mathcal{H}^2(y)+O(\varepsilon^2), \quad \frac{dI_2}{dt} = O(\varepsilon^2).
  \end{align}
  For $I_3$, using the Reynolds transport theorem and observing that the first order time derivative of $R(d(x,t)^2)$ is $R(d(x,t))$ we have
  \begin{align*}
    \frac{dI_3}{dt} = \int_{\Omega_\varepsilon(t)}R(d(x,t))\,dx+\int_{\partial\Omega_\varepsilon(t)}R(d(x,t)^2)V_\varepsilon^N(x,t)\,d\mathcal{H}^2(x).
  \end{align*}
  We apply the change of variables formula \eqref{E:Co-area_NBD} to the first term on the right-hand side of the above equality.
  Then by $R(d(x,t))=R(\rho)=O(\varepsilon)$ and $J(y,t,\rho)=O(1)$ for $d(x,t)=\rho\in(-\varepsilon,\varepsilon)$ with $x\in\Omega_\varepsilon(t)$ to get
  \begin{align*}
    \int_{\Omega_\varepsilon(t)}R(d(x,t))\,dx = \int_{\Gamma(t)}\int_{-\varepsilon}^\varepsilon R(\rho)J(y,t,\rho)\,d\rho\,d\mathcal{H}^2(y) = O(\varepsilon^2).
  \end{align*}
  Moreover, by $R(d(x,t)^2)=O(\varepsilon^2)$ for $x\in\partial\Omega_\varepsilon(t)$ and
  \begin{align*}
    |V_\varepsilon^N(x,t)| = |V_\Gamma^N(\pi(x,t),t)| = O(1), \quad x\in\partial\Omega_\varepsilon(t),
  \end{align*}
  which follows from \eqref{E:N_Vel_Dom} and the fact that $V_\Gamma^N$ is independent of $\varepsilon$, and the change of variables formula \eqref{E:Co-area_BDY} and $J(y,t,\pm\varepsilon)=O(1)$, we have
  \begin{align*}
    \int_{\partial\Omega_\varepsilon(t)}R(d(x,t)^2)V_\varepsilon^N(x,t)\,d\mathcal{H}^2(x) = \sum_{\rho=\pm\varepsilon}\int_{\Gamma(t)}O(\varepsilon^2)J(y,t,\rho)\,d\mathcal{H}^2(y) = O(\varepsilon^2).
  \end{align*}
  Hence $dI_3/dt=O(\varepsilon^2)$.
  From this estimate and \eqref{Pf1:energy_Eu_Limit} it follows that
  \begin{align} \label{E:Int_Kinetic_Limit}
    \frac{d}{dt}\int_{\Omega_\varepsilon(t)}\frac{|u(x,t)|^2}{2}\,dx &= \frac{d}{dt}(I_1+I_2+I_3) \\
    &= 2\varepsilon\frac{d}{dt}\int_{\Gamma(t)}\frac{|v(y,t)|^2}{2}\,d\mathcal{H}^2(y)+O(\varepsilon^2). \notag
  \end{align}
  Let us expand the right-hand side of the energy identity \eqref{E:Energy_Eu_Domain} in $\varepsilon$.
  By the expansion \eqref{E:P_Expand} of the pressure $p$, the relation \eqref{E:N_Vel_Dom}, and the formula \eqref{E:Co-area_BDY},
  \begin{align} \label{Pf2:Energy_Eu_Limit}
    \int_{\partial\Omega_\varepsilon(t)}p(x,t)V_\varepsilon^N(x,t)\,d\mathcal{H}^2(x) = J_1+\varepsilon J_2+O(\varepsilon^2),
  \end{align}
  where
  \begin{align*}
    J_1 &:= \int_{\Gamma(t)}q(y,t)V_\Gamma^N(y,t)\{J(y,t,\varepsilon)-J(y,t,-\varepsilon)\}\,d\mathcal{H}^2(y), \\
    J_2 &:= \int_{\Gamma(t)}q^1(y,t)V_\Gamma^N(y,t)\{J(y,t,\varepsilon)+J(y,t,-\varepsilon)\}\,d\mathcal{H}^2(y).
  \end{align*}
  From \eqref{E:Jacobian} we have
  \begin{align*}
    J(y,t,\varepsilon)-J(y,t,-\varepsilon) &= -2\varepsilon H(y,t)+O(\varepsilon^2), \\
    J(y,t,\varepsilon)+J(y,t,-\varepsilon) &= 2+O(\varepsilon^2).
  \end{align*}
  Hence
  \begin{align*}
    J_1 &= -2\varepsilon\int_{\Gamma(t)}q(y,t)H(y,t)V_\Gamma^N(y,t)\,d\mathcal{H}^2(y), \\
    J_2 &= 2\int_{\Gamma(t)}q^1(y,t)V_\Gamma^N(y,t)\,d\mathcal{H}^2(y)
  \end{align*}
  and applying them to the right-hand side of \eqref{Pf2:Energy_Eu_Limit} we get
  \begin{multline} \label{E:Int_Pressure_Limit}
    \int_{\partial\Omega_\varepsilon(t)}p(x,t)V_\varepsilon^N(x,t)\,d\mathcal{H}^2(x) \\
    = -2\varepsilon\int_{\Gamma(t)}\{q(y,t)H(y,t)-q^1(y,t)\}V_\Gamma^N(y,t)\,d\mathcal{H}^2(y)+O(\varepsilon^2).
  \end{multline}
  Finally, we substitute \eqref{E:Int_Kinetic_Limit} and \eqref{E:Int_Pressure_Limit} for \eqref{E:Energy_Eu_Domain} and divide both sides by $2\varepsilon$ to obtain
  \begin{multline*}
    \frac{d}{dt}\int_{\Gamma(t)}\frac{|v(y,t)|^2}{2}\,d\mathcal{H}^2(y) \\
    = \int_{\Gamma(t)}\{q(y,t)H(y,t)-q^1(y,t)\}V_\Gamma^N(y,t)\,d\mathcal{H}^2(y)+O(\varepsilon).
  \end{multline*}
  Since the left-hand side and the first term on the right-hand side are independent of $\varepsilon$, we conclude that the identity \eqref{E:Energy_Eu_Surface} should be satisfied.
\end{proof}

\subsection{Navier-Stokes equations}

\begin{lemma} \label{L:Energy_NS_Domain}
  Let $u$ and $p$ satisfy the Navier-Stokes equations \eqref{E:NS_Momentum}--\eqref{E:NS_B_Tan} in the moving thin domain $\Omega_\varepsilon(t)$.
  Then we have
  \begin{align} \label{E:Energy_NS_Domain}
    \frac{d}{dt}\int_{\Omega_\varepsilon(t)}\frac{|u|^2}{2}\,dx = -2\mu_0\int_{\Omega_\varepsilon(t)}|D(u)|^2\,dx+\int_{\partial\Omega_\varepsilon(t)}(\sigma\nu_\varepsilon\cdot\nu_\varepsilon)V_\varepsilon^N\,d\mathcal{H}^2.
  \end{align}
  Here $\sigma:=2\mu_0D(u)-pI_3$ denotes the Cauchy stress tensor.
\end{lemma}

The first term on the right-hand side of \eqref{E:Energy_NS_Domain} represents the energy dissipation by viscosity and the second term stands for the rate of work done by the normal component of the stress vector $\sigma\nu_\varepsilon$ on the moving boundary.

\begin{proof}
  By the Reynolds transport theorem (see~\cite{G81bk}) and the equation \eqref{E:NS_Momentum},
  \begin{multline} \label{Pf1:Energy_NS_Domain}
    \frac{d}{dt}\int_{\Omega_\varepsilon(t)}\frac{|u|^2}{2}\,dx = \int_{\Omega_\varepsilon(t)}u\cdot\partial_tu\,dx+\int_{\partial\Omega_\varepsilon(t)}\frac{|u|^2}{2}V_\varepsilon^N\,d\mathcal{H}^2 \\
    = \int_{\Omega_\varepsilon(t)}u\cdot\{-(u\cdot\nabla)u-\nabla p+\mu_0\Delta u\}\,dx+\int_{\partial\Omega_\varepsilon(t)}\frac{|u|^2}{2}V_\varepsilon^N\,d\mathcal{H}^2.
  \end{multline}
  We already computed the integrals of $u\cdot(u\cdot\nabla)u$ and $u\cdot\nabla p$ over $\Omega_\varepsilon(t)$ in the proof of Lemma~\ref{L:Energy_Eu_Domain}, see \eqref{E:Int_Inartia_Domain} and \eqref{E:Int_Pressure_Domain}.
  Let us calculate the integral of $u\cdot\Delta u$.
  Since $\Delta u=2\mathrm{div}\,D(u)$ by the divergence-free condition \eqref{E:NS_Continuity},
  \begin{align*}
    \int_{\Omega_\varepsilon(t)}u\cdot\Delta u\,dx &= 2\int_{\Omega_\varepsilon(t)}u\cdot\mathrm{div}\,D(u)\,dx \\
    &= 2\int_{\partial\Omega_\varepsilon(t)}u\cdot D(u)^T\nu_\varepsilon\,d\mathcal{H}^2-2\int_{\Omega_\varepsilon(t)}\nabla u:D(u)\,dx,
  \end{align*}
  where $F:G:=\mathrm{tr}[F^TG]$ for square matrices $F$ and $G$ of order three.
  In the last line we use the symmetry of the strain rate tensor $D(u)$ and the boundary conditions \eqref{E:NS_B_Norm} and \eqref{E:NS_B_Tan} to get
  \begin{align*}
    u\cdot D(u)^T\nu_\varepsilon = (u\cdot\nu_\varepsilon)(D(u)\nu_\varepsilon\cdot\nu_\varepsilon) = V_\varepsilon^N(D(u)\nu_\varepsilon\cdot\nu_\varepsilon)
  \end{align*}
  on $\partial\Omega_\varepsilon(t)$.
  Also, we easily observe that
  \begin{align*}
    \nabla u:D(u) = (\nabla u)^T:D(u) = |D(u)|^2 \left(=\sum_{i,j=1}^3[D(u)]_{ij}^2\right).
  \end{align*}
  Here $[D(u)]_{ij}$ is the $(i,j)$-entry of $D(u)$, i.e. $[D(u)]_{ij}=(\partial_iu_j+\partial_ju_i)/2$.
  Hence
  \begin{align} \label{E:Int_Viscosity_Domain}
    \int_{\Omega_\varepsilon(t)}u\cdot\Delta u\,dx = 2\int_{\partial\Omega_\varepsilon}(D(u)\nu_\varepsilon\cdot\nu_\varepsilon)V_\varepsilon^N\,d\mathcal{H}^2-2\int_{\Omega_\varepsilon(t)}|D(u)|^2\,dx.
  \end{align}
  Finally we substitute \eqref{E:Int_Inartia_Domain}, \eqref{E:Int_Pressure_Domain}, and \eqref{E:Int_Viscosity_Domain} for \eqref{Pf1:Energy_NS_Domain} to obtain \eqref{E:Energy_NS_Domain}.
\end{proof}

\begin{lemma} \label{L:Energy_NS_Surface}
  Let $v$, $q$, and $q^1$ satisfy the limit equations \eqref{E:NS_Momentum_Limit} and \eqref{E:NS_Continuity_Limit} of the Navier-Stokes equations.
  Suppose that the normal component of $v$ is equal to the outward normal velocity of $\Gamma(t)$, i.e. $v\cdot\nu=V_\Gamma^N$.
  Then we have
  \begin{multline} \label{E:Energy_NS_Surface}
    \frac{d}{dt}\int_{\Gamma(t)}\frac{|v|^2}{2}\,d\mathcal{H}^2 = -2\mu_0\int_{\Gamma(t)}|P_\Gamma D^{tan}(v)P_\Gamma|^2\,d\mathcal{H}^2 \\
    +\int_{\Gamma(t)}(qH-q^1)V_\Gamma^N\,d\mathcal{H}^2.
  \end{multline}
\end{lemma}

The first and second term on the right-hand side of \eqref{E:Energy_NS_Surface} correspond to the energy dissipation of the surface fluid by viscosity and the rate of work done by the moving surface, respectively.

\begin{proof}
  As in the proof of Lemma~\ref{L:Energy_Eu_Surface} we use the Leibniz formula~\cite[Lemma~2.2]{DE07} with velocity field $v$ and the equations \eqref{E:NS_Momentum_Limit} and \eqref{E:NS_Continuity_Limit}:
  \begin{align} \label{Pf1:Energy_NS_Surface}
    \frac{d}{dt}\int_{\Gamma(t)}\frac{|v|^2}{2}\,d\mathcal{H}^2 &= \int_{\Gamma(t)}v\cdot\partial^\bullet_vv\,d\mathcal{H}^2+\int_{\Gamma(t)}\frac{|v|^2}{2}\,\mathrm{div}_\Gamma v\,d\mathcal{H}^2 \\
    &= \int_{\Gamma(t)}v\cdot\{-\nabla_\Gamma q-q^1\nu+2\mu_0\mathrm{div}_\Gamma(P_\Gamma D^{tan}(v)P_\Gamma)\}\,d\mathcal{H}^2. \notag
  \end{align}
  The first two terms in the last line were calculated in the proof of Lemma~\ref{L:Energy_Eu_Surface}, see \eqref{E:Int_Pressure_Tan} and \eqref{E:Int_Pressure_Norm}.
  For the viscous term,
  \begin{align*}
    v\cdot\mathrm{div}_\Gamma(P_\Gamma D^{tan}(v)P_\Gamma) = \mathrm{div}_\Gamma(P_\Gamma D^{tan}(v)P_\Gamma v)-\nabla_\Gamma v:P_\Gamma D^{tan}(v)P_\Gamma.
  \end{align*}
  The integral of the first term on the right-hand side over $\Gamma(t)$ vanishes by Stokes' theorem since $\Gamma(t)$ is closed and $P_\Gamma D^{tan}(v)P_\Gamma v$ is a tangential vector field on $\Gamma(t)$.
  Also, since the matrix $P_\Gamma D^{tan}(v)P_\Gamma$ is symmetric,
  \begin{align*}
    \nabla_\Gamma v:P_\Gamma D^{tan}(v)P_\Gamma = (\nabla_\Gamma v)^T:P_\Gamma D^{tan}(v)P_\Gamma = D^{tan}(v):P_\Gamma D^{tan}(v)P_\Gamma.
  \end{align*}
  Moreover, by the formulas $P_\Gamma^2=P_\Gamma^T=P_\Gamma$ and $E:FG=F^TE:G=EG^T:F$ for square matrices $E$, $F$, and $G$ of order three we obtain
  \begin{align*}
    \nabla_\Gamma v:P_\Gamma D^{tan}(v)P_\Gamma = D^{tan}(v):P_\Gamma D^{tan}(v)P_\Gamma = |P_\Gamma D^{tan}(v)P_\Gamma|^2.
  \end{align*}
  Hence the integral of the inner product of $v$ and $\mathrm{div}_\Gamma(P_\Gamma D^{tan}(v)P_\Gamma)$ is
  \begin{align} \label{E:Int_Viscous_Surface}
    \int_{\Gamma(t)}v\cdot\mathrm{div}_\Gamma(P_\Gamma D^{tan}(v)P_\Gamma)\,d\mathcal{H}^2 = -\int_{\Gamma(t)}|P_\Gamma D^{tan}(v)P_\Gamma|^2\,d\mathcal{H}^2.
  \end{align}
  Applying \eqref{E:Int_Pressure_Tan}, \eqref{E:Int_Pressure_Norm}, and \eqref{E:Int_Viscous_Surface} to \eqref{Pf1:Energy_NS_Surface} we obtain \eqref{E:Energy_NS_Surface}.
\end{proof}

As in the case of the Euler equations, the energy identity \eqref{E:Energy_NS_Surface} on the moving surface can be derived as a thin width limit of that in the moving thin domain \eqref{E:Energy_NS_Domain}.
Let us expand $u$ and $p$ in powers of $d$ as \eqref{E:NS_U_Expand} and \eqref{E:P_Expand} and determine the zeroth order term in $\varepsilon$ of the energy identity \eqref{E:Energy_NS_Domain}.

\begin{theorem} \label{T:Energy_NS_Limit}
  Let $u$ and $p$ satisfy the energy identity \eqref{E:Energy_NS_Domain}.
  Suppose that the velocity field $u$ satisfies the boundary conditions \eqref{E:NS_B_Norm} and \eqref{E:NS_B_Tan}.
  Then the zeroth order term $v$ in the expansion \eqref{E:NS_U_Expand} and the zeroth and first terms $q$ and $q^1$ in the expansion \eqref{E:P_Expand} satisfy the energy identity \eqref{E:Energy_NS_Surface}.
\end{theorem}

\begin{proof}
  The remaining part of the proof is to show that
  \begin{align} \label{E:Int_Viscous_Limit}
    \int_{\Omega_\varepsilon(t)}|D(u)(x,t)|^2\,dx = 2\varepsilon\int_{\Gamma(t)}|(P_\Gamma D^{tan}(v)P_\Gamma)(y,t)|^2\,d\mathcal{H}^2(y)+O(\varepsilon^2)
  \end{align}
  and
  \begin{align} \label{E:Int_Strain_Limit}
    \int_{\partial\Omega_\varepsilon(t)}[(D(u)\nu_\varepsilon\cdot\nu_\varepsilon)V_\varepsilon^N](x,t)\,d\mathcal{H}^2(x) = O(\varepsilon^2)
  \end{align}
  since we already computed other terms in the proof of Theorem~\ref{T:Energy_Eu_Limit}, see \eqref{E:Int_Kinetic_Limit} and \eqref{E:Int_Pressure_Limit}.
  By \eqref{E:NS_Strain_Exp} and \eqref{E:NS_Strain0_Full} in the proof of Theorem~\ref{T:NS_Limit} we have
  \begin{gather*}
    D(u)(x,t) = (P_\Gamma D^{tan}(v)P_\Gamma)(\pi(x,t),t)+d(x,t)S^1(\pi(x,t),t)+R(d(x,t)^2)
  \end{gather*}
  for $x\in \Omega_\varepsilon(t)$
  (Note that to get \eqref{E:NS_Strain0_Full} we only need the boundary conditions \eqref{E:NS_B_Norm} and \eqref{E:NS_B_Tan} for the Navier-Stokes equations.
  See the proof of Theorem~\ref{T:NS_Limit}.)
  Using this expansion and the change of variable formula \eqref{E:Co-area_NBD} we obtain \eqref{E:Int_Viscous_Limit} as
  \begin{multline*}
    \int_{\Omega_\varepsilon(t)}|D(u)(x,t)|^2\,dx \\
    \begin{aligned}
      &= \int_{\Omega_\varepsilon(t)}\{|(P_\Gamma D^{tan}(v)P_\Gamma)(\pi(x,t),t)|^2+R(d(x,t)^2)\}\,dx \\
      &= \int_{\Gamma(t)}\int_{-\varepsilon}^\varepsilon\{|(P_\Gamma D^{tan}(v)P_\Gamma)(y,t)|^2+R(\rho^2)\}J(y,t,\rho)\,d\rho\,d\mathcal{H}^2(y) \\
      &= 2\varepsilon\int_{\Gamma(t)}|(P_\Gamma D^{tan}(v)P_\Gamma)(y,t)|^2\,d\mathcal{H}^2(y)+O(\varepsilon^2).
    \end{aligned}
  \end{multline*}
  Let us show \eqref{E:Int_Strain_Limit}.
  By \eqref{E:N_Vect_Dom} we have
  \begin{align*}
    (D(u)\nu_\varepsilon)(x,t) &= \pm(P_\Gamma D^{tan}(v)P_\Gamma\nu)(\pi(x,t),t)+\varepsilon(S^1\nu)(\pi(x,t),t)+O(\varepsilon^2).
  \end{align*}
  for $x\in\partial\Omega_\varepsilon(t)$ according to $d(x,t)=\pm\varepsilon$ (double-sign corresponds).
  Moreover, the first two terms on the right-hand side vanishes since $P_\Gamma\nu=0$ and $S^1\nu=0$ on $\Gamma(t)$ by \eqref{E:NS_S1nu_Full}.
  (Note that, similarly to the proof of \eqref{E:NS_Strain0_Full}, only the boundary conditions \eqref{E:NS_B_Norm} and \eqref{E:NS_B_Tan} are necessary to show \eqref{E:NS_S1nu_Full}.
  See the proof of Theorem~\ref{T:NS_Limit}.)
  Hence $D(u)\nu_\varepsilon=O(\varepsilon^2)$ on $\partial\Omega_\varepsilon(t)$.
  Applying this estimate and
  \begin{align*}
    |\nu_\varepsilon(x,t)| = 1, \quad |V_\varepsilon^N(x,t)| = |V_\Gamma^N(\pi(x,t),t)| = O(1), \quad x\in\partial\Omega_\varepsilon(t),
  \end{align*}
  where the second relation follows from \eqref{E:N_Vel_Dom} and the fact that $V_\Gamma^N$ is independent of $\varepsilon$, to the left-hand side of \eqref{E:Int_Strain_Limit}, and then using the change of variables formula \eqref{E:Co-area_BDY} and $J(y,t,\pm\varepsilon)=O(1)$, we obtain \eqref{E:Int_Strain_Limit} as
  \begin{align*}
    \int_{\partial\Omega_\varepsilon(t)}[(D(u)\nu_\varepsilon\cdot\nu_\varepsilon)V_\varepsilon^N](x,t)\,d\mathcal{H}^2(x) = \sum_{\rho=\pm\varepsilon}\int_{\Gamma(t)}O(\varepsilon^2)J(y,t,\rho)\,d\mathcal{H}^2 = O(\varepsilon^2).
  \end{align*}
  Now we substitute \eqref{E:Int_Kinetic_Limit}, \eqref{E:Int_Pressure_Limit}, \eqref{E:Int_Viscous_Limit}, and \eqref{E:Int_Strain_Limit} for the energy identity \eqref{E:Energy_NS_Domain} and divide both sides by $2\varepsilon$ to obtain
  \begin{multline*}
    \frac{d}{dt}\int_{\Gamma(t)}\frac{|v(y,t)|^2}{2}\,d\mathcal{H}^2(y) = -2\mu_0\int_{\Gamma(t)}|(P_\Gamma D^{tan}(v)P_\Gamma)(y,t)|^2\,d\mathcal{H}^2(y) \\
    +\int_{\Gamma(t)}(qH-q^1)(y,t)V_\Gamma^N(y,t)\,d\mathcal{H}^2(y)+O(\varepsilon).
  \end{multline*}
  Since all terms except for $O(\varepsilon)$ are independent of $\varepsilon$, we conclude that the energy identity \eqref{E:Energy_NS_Surface} must be satisfied.
\end{proof}

\begin{remark} \label{R:Energy_NS_Limit}
  The assumption in Theorem~\ref{T:Energy_NS_Limit} that the boundary conditions \eqref{E:NS_B_Norm} and \eqref{E:NS_B_Tan} are satisfied is necessary to deal with integrals including the strain rate tensor $D(u)$.
  Note that, contrary to the case of the Navier-Stokes equations (Theorem~\ref{T:Energy_NS_Limit}), we do not need even the impermeable boundary condition \eqref{E:Eu_Boundary} to derive the thin width limit of the energy identity of the Euler equations in the moving thin domain, see Theorem~\ref{T:Energy_Eu_Limit}.
\end{remark}

\begin{appendices}

\section{Elementary calculations of various quantities on surfaces} \label{S:Appendix_A}
In this appendix we prove elementary facts on various quantities and differential operators on a surface given in Section~\ref{S:Preliminaries}.
Until the end of the proof of Lemma~\ref{L:Co-area} we fix and suppress $t\in[0,T]$.

\begin{proof}[Proof of Lemma~\ref{L:Weingarten}]
  Since $|\nu|^2=1$ on $\Gamma$, we have
  \begin{align*}
    0 = \nabla_\Gamma|\nu|^2 = 2(\nabla_\Gamma\nu)\nu = -2A\nu \quad\text{on}\quad \Gamma,
  \end{align*}
  which implies \eqref{E:W_Normal}.
  The formula \eqref{E:W_OrthProj} is an immediate consequence of \eqref{E:W_Normal}.
  Now let us prove \eqref{E:W_Hessian}.
  Let $\tilde{\nu}$ be an extension of $\nu$ to $N$.
  By \eqref{E:Nabla_Dist} and $\tilde{\nu}|_\Gamma=\nu$ we have
  \begin{align} \label{Pf1:Weingarten}
    \nabla^2d(x) = \nabla\pi(x)(\nabla\tilde{\nu})(\pi(x)), \quad x\in N.
  \end{align}
  Moreover, we differentiate both sides of \eqref{E:N_Coordinate} and apply \eqref{E:Nabla_Dist} to get
  \begin{align*}
    \nabla\pi(x) = P_\Gamma(\pi(x))-d(x)\nabla\pi(x)(\nabla\tilde{\nu})(\pi(x)), \quad x\in N.
  \end{align*}
  In particular, if $x=y\in\Gamma$ then $d(x)=0$, $\pi(x)=y$ and thus
  \begin{align*}
    \nabla\pi(y) = P_\Gamma(y), \quad y\in\Gamma.
  \end{align*}
  Applying this formula to \eqref{Pf1:Weingarten} with $x=y\in\Gamma$ we obtain \eqref{E:W_Hessian}.
\end{proof}

\begin{proof}[Proof of Lemma~\ref{L:Exchange_Derivative}]
  Let $f$ be a function on $\Gamma$ and $\tilde{f}$ its extension to $N$ satisfying $\tilde{f}|_\Gamma=f$.
  For $j=1,2,3$, by \eqref{E:Nabla_Dist} and the definition of the tangential derivative operators we have
  \begin{align*}
    \partial_j^{tan}f(y) = \sum_{l=1}^3\{\delta_{jl}-\partial_jd(y)\partial_ld(y)\}\partial_l\tilde{f}(y), \quad y\in\Gamma.
  \end{align*}
  From now on we suppress the argument $y$.
  By the above formula we have
  \begin{align*}
    \partial_i^{tan}\partial_j^{tan}f &= \sum_{k,l=1}^3\{\delta_{ik}-(\partial_id)(\partial_kd)\}\partial_k\Bigl[\{\delta_{jl}-(\partial_jd)(\partial_ld)\}\partial_l\tilde{f}\,\Bigr] \\
    &= \alpha_1+\alpha_2+\alpha_3
  \end{align*}
  for $i,j=1,2,3$, where
  \begin{align*}
    \alpha_1 &:= \sum_{k,l=1}^3\{\delta_{ik}-(\partial_id)(\partial_kd)\}\{\delta_{jl}-(\partial_jd)(\partial_ld)\}\partial_k\partial_l\tilde{f}, \\
    \alpha_2 &:= -\sum_{k,l=1}^3\{\delta_{ik}-(\partial_id)(\partial_kd)\}(\partial_k\partial_jd)(\partial_ld)\partial_l\tilde{f}, \\
    \alpha_3 &:= -\sum_{k,l=1}^3\{\delta_{ik}-(\partial_id)(\partial_kd)\}(\partial_jd)(\partial_k\partial_ld)\partial_l\tilde{f}.
  \end{align*}
  Similarly, we have $\partial_j^{tan}\partial_i^{tan}f=\beta_1+\beta_2+\beta_3$, where
  \begin{align*}
    \beta_1 &:= \sum_{k,l=1}^3\{\delta_{jl}-(\partial_jd)(\partial_ld)\}\{\delta_{ik}-(\partial_id)(\partial_kd)\}\partial_l\partial_k\tilde{f}, \\
    \beta_2 &:= -\sum_{k,l=1}^3\{\delta_{jl}-(\partial_jd)(\partial_ld)\}(\partial_l\partial_id)(\partial_kd)\partial_k\tilde{f}, \\
    \beta_3 &:= -\sum_{k,l=1}^3\{\delta_{jl}-(\partial_jd)(\partial_ld)\}(\partial_id)(\partial_l\partial_kd)\partial_k\tilde{f}.
  \end{align*}
  From $\partial_k\partial_l\tilde{f}=\partial_l\partial_k\tilde{f}$ it immediately follows that $\alpha_1=\beta_1$.
  Since $\partial_k\partial_jd=\partial_j\partial_kd$,
  \begin{align*}
    \alpha_2 &= -(\nabla d\cdot\nabla\tilde{f})\left\{\partial_i\partial_jd-(\partial_id)\sum_{k=1}^3(\partial_kd)(\partial_j\partial_kd)\right\} \\
    &= -(\nabla d\cdot\nabla\tilde{f})\left\{\partial_i\partial_jd-(\partial_id)\partial_j\left(\frac{|\nabla d|^2}{2}\right)\right\} \\
    &= -(\nabla d\cdot\nabla\tilde{f})\partial_i\partial_jd.
  \end{align*}
  Here the last equality follows from $|\nabla d|^2=1$ on $N$.
  By the same calculation we have $\beta_2=-(\nabla d\cdot\nabla\tilde{f})\partial_j\partial_id$.
  Hence $\alpha_2=\beta_2$ by $\partial_i\partial_jd=\partial_j\partial_id$.
  For $\alpha_3$ and $\beta_3$,
  \begin{align*}
    \alpha_3 &= -\Bigl[P_\Gamma(\nabla^2d)\nabla\tilde{f}\,\Bigr]_i\partial_jd = [A\nabla_\Gamma f]_i\nu_j, \\
    \beta_3 &= -\Bigl[P_\Gamma(\nabla^2d)\nabla\tilde{f}\,\Bigr]_j\partial_id = [A\nabla_\Gamma f]_j\nu_i
  \end{align*}
  by \eqref{E:Nabla_Dist}, \eqref{E:W_OrthProj}, \eqref{E:W_Hessian}, and the definition of the tangential gradient operator. (Note that we calculate values of functions at $y\in\Gamma(t).$)
  Therefore, we obtain
  \begin{align*}
    \partial_i^{tan}\partial_j^{tan}f-\partial_j^{tan}\partial_i^{tan}f &= (\alpha_1+\alpha_2+\alpha_3)-(\beta_1+\beta_2+\beta_3) \\
    &= [A\nabla_\Gamma f]_i\nu_j-[A\nabla_\Gamma f]_j\nu_i,
  \end{align*}
  that is, the formula \eqref{E:Exchange_Derivative} holds.
\end{proof}

\begin{proof}[Proof of Lemma~\ref{L:Viscous_General}]
  Let $v$ be a general vector field on $\Gamma$ which may have a nonzero normal component.
  Since $P_\Gamma\nabla_\Gamma v=\nabla_\Gamma v$ and $(\nabla_\Gamma v)^TP_\Gamma=(\nabla_\Gamma v)^T$ we have
  \begin{align} \label{Pf1:Viscous_General}
    2\mathrm{div}_\Gamma(P_\Gamma D^{tan}(v)P_\Gamma) = \mathrm{div}_\Gamma\bigl((\nabla_\Gamma v)P_\Gamma\bigr)+\mathrm{div}_\Gamma\bigl(P_\Gamma(\nabla_\Gamma v)^T\bigr).
  \end{align}
  Let us calculate each term on the right-hand side.
  For $i,j=1,2,3$ the $(i,j)$-entry of $(\nabla_\Gamma v)P_\Gamma$ is of the form
  \begin{align*}
    \bigl[(\nabla_\Gamma v)P_\Gamma\bigr]_{ij} = \sum_{k=1}^3(\partial_i^{tan}v_k)(\delta_{jk}-\nu_j\nu_k).
  \end{align*}
  Thus the $j$-th component of $\mathrm{div}_\Gamma\bigl((\nabla_\Gamma v)P_\Gamma\bigr)$ is
  \begin{align*}
    \Bigl[\mathrm{div}_\Gamma\bigl((\nabla_\Gamma v)P_\Gamma\bigr)\Bigr]_j = \sum_{i=1}^3\partial_i^{tan}\bigl[(\nabla_\Gamma v)P_\Gamma\bigr]_{ij} = \alpha_1+\alpha_2+\alpha_3,
  \end{align*}
  where
  \begin{align*}
    \alpha_1 &:= \sum_{i,k=1}^3\{(\partial_i^{tan})^2v_k\}(\delta_{jk}-\nu_j\nu_k), \\
    \alpha_2 &:= -\sum_{i,k=1}^3(\partial_i^{tan}v_k)(\partial_i^{tan}\nu_j)\nu_k = \sum_{i,k=1}^3(\partial_i^{tan}v_k)A_{ij}\nu_k, \\
    \alpha_3 &:= -\sum_{i,k=1}^3(\partial_i^{tan}v_k)\nu_j(\partial_i^{tan}\nu_k) = \sum_{i,k=1}^3(\partial_i^{tan}v_k)\nu_jA_{ik}.
  \end{align*}
  Here $A_{ij}$ is the $(i,j)$-entry of the Weingarten map $A=-\nabla_\Gamma\nu$.
  By the definitions of $\Delta_\Gamma$ and $P_\Gamma$ we have $\alpha_1=\bigl[P_\Gamma(\Delta_\Gamma v)\bigr]_j$, where $\Delta_\Gamma$ applies each component of the vector field $v$.
  Also, since $A$ is symmetric,
  \begin{align*}
    \alpha_2 = \sum_{i,k=1}^3A_{ji}(\partial_i^{tan}v_k)\nu_k = \bigl[A(\nabla_\Gamma v)\nu\bigr]_j.
  \end{align*}
  Similarly, we have $\alpha_3=\mathrm{tr}[A\nabla_\Gamma v]\nu_j$.
  Therefore, the equality
  \begin{align*}
    \Bigl[\mathrm{div}_\Gamma\bigl((\nabla_\Gamma v)P_\Gamma\bigr)\Bigr]_j = \bigl[P_\Gamma(\Delta_\Gamma v)\bigr]_j+\bigl[A(\nabla_\Gamma v)\nu\bigr]_j+\mathrm{tr}[A\nabla_\Gamma v]\nu_j
  \end{align*}
  holds for each $j=1,2,3$, which means that
  \begin{align} \label{Pf2:Viscous_General}
    \mathrm{div}_\Gamma\bigl((\nabla_\Gamma v)P_\Gamma\bigr) = P_\Gamma(\Delta_\Gamma v)+A(\nabla_\Gamma v)\nu+\mathrm{tr}[A\nabla_\Gamma v]\nu.
  \end{align}
  Calculations of the second term $\mathrm{div}_\Gamma\bigl(P_\Gamma(\nabla_\Gamma v)^T\bigr)$ are more complicated.
  Since
  \begin{align*}
    \Bigl[P_\Gamma(\nabla_\Gamma v)^T\Bigr]_{ij} = \sum_{k=1}^3(\delta_{ik}-\nu_i\nu_k)\partial_j^{tan}v_k,
  \end{align*}
  we have $\Bigl[\mathrm{div}_\Gamma\bigl(P_\Gamma(\nabla_\Gamma v)^T\bigr)\Bigr]_j=\beta_1+\beta_2+\beta_3$, where
  \begin{align*}
    \beta_1 &:= -\sum_{i,k=1}^3(\partial_i^{tan}\nu_i)\nu_k\partial_j^{tan}v_k = \sum_{i,k=1}^3A_{ii}\nu_k\partial_j^{tan}v_k, \\
    \beta_2 &:= -\sum_{i,k=1}^3\nu_i(\partial_i^{tan}\nu_k)\partial_j^{tan}v_k = \sum_{i,k=1}^3\nu_iA_{ik}\partial_j^{tan}v_k, \\
    \beta_3 &:= \sum_{i,k=1}^3(\delta_{ik}-\nu_i\nu_k)\partial_i^{tan}\partial_j^{tan}v_k.
  \end{align*}
  By the definition of the mean curvature,
  \begin{align*}
    \beta_1 = \mathrm{tr}[A]\sum_{k=1}^3(\partial_j^{tan}v_k)\nu_k = H\bigl[(\nabla_\Gamma v)\nu\bigr]_j.
  \end{align*}
  Since $A_{ik}=A_{ki}$ and $A\nu=0$,
  \begin{align*}
    \beta_2 = \sum_{i,k=1}^3(\partial_j^{tan}v_k)A_{ki}\nu_i = -\bigl[(\nabla_\Gamma v)A\nu\bigr]_j = 0.
  \end{align*}
  For $\beta_3$ we have
  \begin{align*}
    \beta_3 = \sum_{i=1}^3\partial_i^{tan}\partial_j^{tan}v_i-\sum_{k=1}^3\nu_k\{\nu\cdot\nabla_\Gamma(\partial_j^{tan}v_k)\}.
  \end{align*}
  The second term on the right-hand side vanishes since $\nu\cdot\nabla_\Gamma(\partial_j^{tan}v_k)=0$ for each $j$ and $k$.
  We apply \eqref{E:Exchange_Derivative} to the first term to get
  \begin{align*}
    \beta_3 &= \sum_{i=1}^3\partial_j^{tan}\partial_i^{tan}v_i+\sum_{i=1}^3[A\nabla_\Gamma v_i]_i\nu_j-\sum_{i=1}^3[A\nabla_\Gamma v_i]_j\nu_i \\
    &= \partial_j^{tan}(\mathrm{div}_\Gamma v)+\mathrm{tr}[A\nabla_\Gamma v]\nu_j-\bigl[A(\nabla_\Gamma v)\nu\bigr]_j.
  \end{align*}
  Therefore, it follows that
  \begin{align*}
    \Bigl[\mathrm{div}_\Gamma\bigl(P_\Gamma(\nabla_\Gamma v)^T\bigr)\Bigr]_j = \partial_j^{tan}(\mathrm{div}_\Gamma v)+\bigl[(HI_3-A)(\nabla_\Gamma v)\nu\bigr]_j+\mathrm{tr}[A\nabla_\Gamma v]\nu_j
  \end{align*}
  for each $j=1,2,3$ and thus
  \begin{align} \label{Pf3:Viscous_General}
    \mathrm{div}_\Gamma\bigl(P_\Gamma(\nabla_\Gamma v)^T\bigr) = \nabla_\Gamma(\mathrm{div}_\Gamma v)+(HI_3-A)(\nabla_\Gamma v)\nu+\mathrm{tr}[A\nabla_\Gamma v]\nu.
  \end{align}
  Substituting \eqref{Pf2:Viscous_General} and \eqref{Pf3:Viscous_General} for \eqref{Pf1:Viscous_General} we obtain the formula \eqref{E:Viscous_General}.
\end{proof}

\begin{proof}[Proof of Lemma~\ref{L:Co-area}]
  For $\rho\in[-\varepsilon,\varepsilon]$ let $\Gamma_\rho:=\{x\in\mathbb{R}^3\mid d(x)=\rho\}$ be a level-set surface of $\Gamma$.
  Suppose that the change of variables formula
    \begin{align} \label{E:Level-set_Integral}
    \int_{\Gamma_\rho}f(z)\,d\mathcal{H}^2(z) = \int_{\Gamma}f(y+\rho\nu(y))J(y,\rho)\,d\mathcal{H}^2(y)
  \end{align}
  holds for each $\rho\in[-\varepsilon,\varepsilon]$.
  Then \eqref{E:Co-area_NBD} and \eqref{E:Co-area_BDY} follow from this formula and
  \begin{align*}
    \int_{\Omega_\varepsilon}f(x)\,dx = \int_{-\varepsilon}^\varepsilon\left(\int_{\Gamma_\rho}f(z)\,d\mathcal{H}^2(z)\right)d\rho,
  \end{align*}
  which is the well-known co-area formula (see e.g.~\cite[Theorem~2.9]{DE13}), and
  \begin{align*}
    \int_{\partial\Omega_\varepsilon}f(x)\,d\mathcal{H}^2(x) = \int_{\Gamma_\varepsilon}f(z)\,d\mathcal{H}^2(z)+\int_{\Gamma_{-\varepsilon}}f(z)\,d\mathcal{H}^2(z).
  \end{align*}
  Let us prove \eqref{E:Level-set_Integral}.
  Since $\Gamma$ is compact, we may take finitely many open subsets $U_k$ of $\mathbb{R}^2$ and local parametrizations $\mu_k\colon U_k\to\Gamma$ ($k=1,\dots,N$) such that $\{\mu_k(U_k)\}_{k=1}^N$ is an open covering of $\Gamma$.
  Let $\{\varphi_k\}_{k=1}^N$ be a partition of unity of $\Gamma$ subordinate to the covering $\{\mu_k(U_k)\}_{k=1}^N$ and for each $\rho\in[-\varepsilon,\varepsilon]$ and $k=1,\dots,N$ set
  \begin{align*}
    \mu_k^\rho(s) := \mu_k(s)+\rho\nu(\mu_k(s)), \quad \varphi_k^\rho(\mu_k^\rho(s)) := \varphi_k(\mu_k(s)), \quad s\in U_k.
  \end{align*}
  Then $\mu_k^\rho\colon U_k\to\Gamma_\rho$ is a local parametrization of $\Gamma_\rho$ whose domain is the same as that of $\mu_k$ and $\{\mu_k^\rho(U_k)\}_{k=1}^N$ is an open covering of $\Gamma_\rho$.
  Moreover, $\{\varphi_k^\rho\}_{k=1}$ is a partition of unity of $\Gamma_\rho$ subordinate to the covering $\{\mu_k^\rho(U_k)\}_{k=1}^N$.
  By these partitions of unity and the definition of integrals over a surface, the proof of \eqref{E:Level-set_Integral} reduces to showing that, for any local parametrization $\mu\colon U\to\Gamma$ with an open subset $U$ of $\mathbb{R}^2$ and $\mu^\rho\colon U\to\Gamma_\rho$ given by $\mu^\rho(s):=\mu(s)+\rho\nu(\mu(s))$, $s\in U$, the formula
  \begin{align} \label{E:Measure_Transform}
    \sqrt{\det\theta^\rho(s)} = J(\mu(s),\rho)\sqrt{\det\theta(s)}, \quad s\in U
  \end{align}
  holds.
  Here $\theta$ is a square matrix of order two given by $\theta:=\nabla'\mu(\nabla'\mu)^T$, where
  \begin{align*}
    \nabla'\mu :=
    \begin{pmatrix}
      \partial'_1\mu_1 & \partial'_1\mu_2 & \partial'_1\mu_3 \\
      \partial'_2\mu_1 & \partial'_2\mu_2 & \partial'_2\mu_3
    \end{pmatrix}
    \quad \left(\partial'_i := \frac{\partial}{\partial s_i}\right),
  \end{align*}
  and $\theta^\rho:=\nabla'\mu^\rho(\nabla'\mu^\rho)^T$.
  We define square matrices $M$ and $M^\rho$ of order three as
  \begin{align*}
    M(s) :=
    \begin{pmatrix}
      \nabla'\mu(s) \\
      [\nu(\mu(s))]^T
    \end{pmatrix}, \quad
    M^\rho(s) :=
    \begin{pmatrix}
      \nabla'\mu^\rho(s) \\
      [\nu(\mu(s))]^T
    \end{pmatrix}.
  \end{align*}
  Here we see $\nu(\mu(s))$ as a three-dimensional row vector.
  In the following argument, we sometimes suppress the argument $s$ and abbreviate $\nu(\mu(s))$ to $\nu$.
  For $i=1,2$ the $i$-th component of $\nabla'\mu(s)\nu(\mu(s))\in\mathbb{R}^2$ is $\partial_i\mu(s)\cdot\nu(\mu(s))=0$ since $\partial_i\mu(s)$ is tangent to $\Gamma$ at $\mu(s)$.
  Therefore, $(\nabla'\mu)\nu=0$ and
  \begin{align*}
    MM^T =
    \begin{pmatrix}
      \nabla'\mu(\nabla'\mu)^T & (\nabla'\mu)\nu \\
      [(\nabla'\mu)\nu]^T & |\nu|^2
    \end{pmatrix}
    =
    \begin{pmatrix}
      \theta & 0 \\
      0 & 1
    \end{pmatrix},
  \end{align*}
  which implies $\det\theta=\det(MM^T)=(\det M)^2$.
  On the other hand, since
  \begin{align*}
    \mu^\rho(s) = \mu(s)+\rho\nu(\mu(s)) = \mu(s)+\rho\nabla d(\mu(s))
  \end{align*}
  by \eqref{E:Nabla_Dist} and thus
  \begin{align*}
    \nabla'\mu^\rho(s) = \nabla'\mu(s)\{I_3+\rho\nabla^2d(\mu(s))\} = \nabla'\mu(s)\{I_3-\rho A(\mu(s))\}
  \end{align*}
  by \eqref{E:W_Hessian}, we have $\nabla'\mu^\rho(s)\nu(\mu(s))=0$ by $\nabla'\mu(s)\nu(\mu(s))=0$ and \eqref{E:W_Normal}.
  Hence as in the case of $\theta$ and $M$ we have $\det\theta^\rho=(\det M^\rho)^2$.
  Moreover, by \eqref{E:W_Normal} and the symmetry of the matrix $I_3-\rho A$,
  \begin{align*}
    M^\rho =
    \begin{pmatrix}
      (\nabla'\mu)(I_3-\rho A) \\
      \nu^T
    \end{pmatrix}
    =
    \begin{pmatrix}
      \nabla'\mu \\
      \nu^T
    \end{pmatrix}(I_3-\rho A)
    = M(I_3-\rho A).
  \end{align*}
  Hence we get
  \begin{align*}
    \det\theta^\rho = (\det M^\rho)^2 = \{\det M\cdot\det(I_3-\rho A)\}^2 = \{\det(I_3-\rho A)\}^2\det\theta.
  \end{align*}
  Finally we observe that the Weingarten map $A$ has the eigenvalues $0$, $\kappa_1$, and $\kappa_2$ and thus
  \begin{align*}
    \det\{I_3-\rho A(\mu(s))\} &= 1\cdot\{1-\rho\kappa_1(\mu(s))\}\cdot\{1-\rho\kappa_2(\mu(s))\} \\
    &= J(\mu(s),\rho) \quad (>0 \quad\text{for sufficiently small $\rho$})
  \end{align*}
  to obtain the formula \eqref{E:Measure_Transform}.
\end{proof}

Now let us return to the moving surface $\Gamma(t)$ and prove Lemmas~\ref{L:Diff_Composite} and~\ref{L:Div_Matrix}.

\begin{proof}[Proof of Lemma~\ref{L:Diff_Composite}]
  As in the proof of Theorem~\ref{T:Eu_Limit} we use the abbreviations \eqref{E:Abbrev}.
  Let $f$ be a function on $S_T$ and $\tilde{f}$ an arbitrary extension of $f$ to $N_T$ satisfying $\tilde{f}|_{S_T}=f$.
  For $(x,t)\in Q_{\varepsilon,T}$ we have $f(\pi,t)=\tilde{f}(\pi,t)$ by $\pi=\pi(x,t)\in\Gamma(t)$ and thus
  \begin{align*}
    \nabla\bigl(f(\pi,t)\bigr) &= \nabla\pi(x,t)\nabla\tilde{f}(\pi,t), \\
    \partial_t\bigl(f(\pi,t)\bigr) &= \partial_t\tilde{f}(\pi,t)+(\partial_t\pi(x,t)\cdot\nabla)\tilde{f}(\pi,t).
  \end{align*}
  Hence it is sufficient for \eqref{E:Nabla_Composite} and \eqref{E:Dt_Composite} to show that
  \begin{align}
    \nabla\pi(x,t) &= P_\Gamma(\pi,t)+d(x,t)A(\pi,t)+R(d^2), \label{E:Nabla_Proj} \\
    \partial_t\pi(x,t) &= V_\Gamma^N(\pi,t)\nu(\pi,t)+d(x,t)\nabla_\Gamma V_\Gamma^N(\pi,t)+R(d), \label{E:Dt_Proj}
  \end{align}
  since
  \begin{gather*}
    A\nabla\tilde{f} = AP_\Gamma\nabla\tilde{f} = A\nabla_\Gamma f, \\
    \partial_t\tilde{f}+(V_\Gamma^N\nu\cdot\nabla)\tilde{f} = \partial^\circ f, \quad (\nabla_\Gamma V_\Gamma^N\cdot\nabla)\tilde{f} = (\nabla_\Gamma V_\Gamma^N\cdot\nabla_\Gamma)f
  \end{gather*}
  on $\Gamma(t)$ by the definition of the tangential gradient, \eqref{E:W_OrthProj}, and \eqref{E:Material_Euler_Coordinate} with $v=V_\Gamma^N\nu$.
  By $\pi(x,t) = x-d(x,t)\nabla d(x,t)$ and \eqref{E:Nabla_Dist} we have
  \begin{align*}
    \nabla\pi(x,t) &= I_3-\nabla d(x,t)\otimes\nabla d(x,t)-d(x,t)\nabla^2d(x,t) \\
    &= P_\Gamma(\pi,t)-d(x,t)\nabla^2d(x,t).
  \end{align*}
  Also, we expand $\nabla^2d$ in powers of $d$ and apply \eqref{E:W_Hessian} to obtain
  \begin{align*}
    \nabla^2d(x,t) = \nabla^2d(\pi,t)+R(d) = -A(\pi,t)+R(d).
  \end{align*}
  Hence \eqref{E:Nabla_Proj} follows.
  Similarly, we differentiate $\pi(x,t)=x-d(x,t)\nabla d(x,t)$ with respect to $t$ and apply \eqref{E:Nabla_Dist} and \eqref{E:Dt_Dist_NBD} to get
  \begin{align*}
    \partial_t\pi(x,t) = V_\Gamma^N(\pi,t)\nu(\pi,t)-d(x,t)\partial_t\nabla d(x,t).
  \end{align*}
  Moreover, by $\partial_t\nabla d=\nabla\partial_td$, \eqref{E:Dt_Dist_NBD}, and \eqref{E:Nabla_Proj},
  \begin{align*}
    \partial_t\nabla d(x,t) = -\nabla\bigl(V_\Gamma^N(\pi,t)\bigr) = -\nabla\pi(x,t)\nabla\widetilde{V}_\Gamma^N(\pi,t) = -\nabla_\Gamma V_\Gamma^N(\pi,t)+R(d),
  \end{align*}
  where $\widetilde{V}_\Gamma^N$ is an extension of $V_\Gamma^N$ to $N_T$ with $\widetilde{V}_\Gamma^N|_{S_T}=V_\Gamma^N$.
  Applying this to the above equality for $\partial_t\pi$ we obtain \eqref{E:Dt_Proj}.
\end{proof}

\begin{proof}[Proof of Lemma~\ref{L:Div_Matrix}]
  We use the abbreviations \eqref{E:Abbrev}.
  For $i,j=1,2,3$, let $M_{ij}$ be the $(i,j)$-entry of a square matrix $M$ of order three.
  We differentiate both sides of $D_{ij}(x) = S_{ij}(\pi)+d(x,t)S_{ij}^1(\pi)+R(d^2)$ with respect to $x_i$ and apply \eqref{E:Nabla_Proj} to get
  \begin{align*}
    \partial_iD_{ij}(x) = \partial_i^{tan}S_{ij}(\pi)+S_{ij}^1(\pi)\partial_id(x,t)+R(d).
  \end{align*}
  Therefore, the $j$-th component of $\mathrm{div}\,D(x)$ is
  \begin{align*}
    [\mathrm{div}\,D(x)]_j &= \sum_{i=1}^3\partial_iD_{ij}(x) = \sum_{i=1}^3\{\partial_i^{tan}S_{ij}(\pi)+S_{ij}^1(\pi)\partial_id(x,t)\}+R(d(x,t)) \\
    &= [\mathrm{div}_\Gamma S(\pi)]_j+\Bigl[\bigl(S^1(\pi)\bigr)^T\nabla d(x,t)\Bigr]_j+R(d)
  \end{align*}
  and \eqref{E:Div_Matrix} follows by \eqref{E:Nabla_Dist}.
\end{proof}

\section{Comparison of vector Laplacians} \label{S:Appendix_B}
The purpose of this appendix is to give a proof of the formula \eqref{E:Bochner_Beltrami} in Lemma~\ref{L:Riemannian_Connection}.
Main tools for the proof are the Gauss formula \eqref{E:Gauss_Formula} and
\begin{align} \label{E:Bochner_Connection}
  \Delta_BX = \mathrm{tr}\overline{\nabla}^2X = \sum_{i=1}^2\Bigl(\overline{\nabla}_i\overline{\nabla}_iX-\overline{\nabla}_{\overline{\nabla}_ie_i}X\Bigr) \quad\text{on}\quad \Gamma
\end{align}
for any tangential vector field $X$ on $\Gamma$, where $\{e_1,e_2\}$ denotes a local orthonormal frame of $T\Gamma$ (i.e. an orthonormal basis of the tangent plane of $\Gamma$ defined on a relative open subset of $\Gamma$) and $\overline{\nabla}_i:=\overline{\nabla}_{e_i}$ (for a proof of \eqref{E:Bochner_Connection} see~\cite[Proposition~34]{P06bk} and~\cite[Proposition~2.1 in Appendix~C]{T11bk}).
Hereafter all calculations are carried out on the surface $\Gamma$.

We fix coordinates of $\mathbb{R}^3$ and write $x_j$ ($j=1,2,3$) for the $j$-th component of a point $x\in\mathbb{R}^3$ under this fixed coordinates.
Let $X=(X_1,X_2,X_3)$ be a tangential vector field on $\Gamma$ and $\{e_1,e_2\}$ be a local orthonormal frame of $T\Gamma$.
For $i=1,2$, by the Gauss formula \eqref{E:Gauss_Formula_Another} and the fact that $\overline{\nabla}_iX$ is tangential we have
\begin{align*}
  \overline{\nabla}_iX = (e_i\cdot\nabla_\Gamma)X-(AX\cdot e_i)\nu = P_\Gamma\{(e_i\cdot\nabla_\Gamma)X\}.
\end{align*}
Here the second equality follows from $P_\Gamma\nu=0$.
Hence
\begin{align*}
  \overline{\nabla}_i\overline{\nabla}_i X &= P_\Gamma\bigl[(e_i\cdot\nabla_\Gamma)\{(e_i\cdot\nabla_\Gamma)X-(AX\cdot e_i)\nu\}\bigr] \\
  &= P_\Gamma\bigl[(e_i\cdot\nabla_\Gamma)\{(e_i\cdot\nabla_\Gamma)X\}\bigr]-(AX\cdot e_i)P_\Gamma\{(e_i\cdot\nabla_\Gamma)\nu\},
\end{align*}
where we used $P_\Gamma\nu=0$ again in the second equality.
By setting $e_i=(e_i^1,e_i^2,e_i^3)$ the $j$-th component of the vector $(e_i\cdot\nabla_\Gamma)\{(e_i\cdot\nabla_\Gamma)X\}$ ($j=1,2,3$) is of the form
\begin{align*}
  \sum_{k,l=1}^3e_i^k\partial_k^{tan}(e_i^l\partial_l^{tan}X_j) &= \sum_{k,l=1}^3\{e_i^ke_i^l\partial_k^{tan}\partial_l^{tan}X_j+e_i^k(\partial_k^{tan}e_i^l)\partial_l^{tan}X_j\} \\
  &= \mathrm{tr}\bigl[(e_i\otimes e_i)\nabla_\Gamma^2 X_j\bigr]+\{(e_i\cdot\nabla_\Gamma)e_i\}\cdot\nabla_\Gamma X_j.
\end{align*}
Also, by the symmetry of the Weingarten map $A=-\nabla_\Gamma\nu$,
\begin{align*}
  [(e_i\cdot\nabla_\Gamma)\nu]_j = \sum_{k=1}^3e_i^k\partial_k^{tan}\nu_j = -\sum_{k=1}^3e_i^kA_{kj} = -[Ae_i]_j.
\end{align*}
By these equalities and \eqref{E:W_OrthProj} the $j$-th component of $\overline{\nabla}_i\overline{\nabla}_iX$ is
\begin{multline} \label{E:Second_Covariant_First}
  \Bigl[\overline{\nabla}_i\overline{\nabla}_iX\Bigr]_j = \sum_{k=1}^3[P_\Gamma]_{jk}\Bigl(\mathrm{tr}\bigl[(e_i\otimes e_i)\nabla_\Gamma^2 X_k\bigr]+\{(e_i\cdot\nabla_\Gamma)e_i\}\cdot\nabla_\Gamma X_k\Bigr) \\
  +(AX\cdot e_i)[Ae_i]_j.
\end{multline}
On the other hand, $\overline{\nabla}_{\overline{\nabla}_ie_i}X$ is of the form
\begin{align*}
  \overline{\nabla}_{\overline{\nabla}_ie_i}X = P_\Gamma\Bigl\{\Bigl(\overline{\nabla}_ie_i\cdot\nabla_\Gamma\Bigr)X\Bigr\} = P_\Gamma\Bigl(\bigl[\{P_\Gamma(e_i\cdot\nabla_\Gamma)e_i\}\cdot\nabla_\Gamma\bigl]X\Bigr)
\end{align*}
and, since $\{(P_\Gamma F)\cdot\nabla_\Gamma\}G = (F\cdot\nabla_\Gamma)G$ holds for (not necessarily tangential) vector fields $F$ and $G$ on $\Gamma$ we have
\begin{align} \label{E:Second_Covariant_Second}
  \Bigl[\overline{\nabla}_{\overline{\nabla}_ie_i}X\Bigr]_j = \sum_{k=1}^3[P_\Gamma]_{jk}\Bigl(\{(e_i\cdot\nabla_\Gamma)e_i\}\cdot\nabla_\Gamma X_k\Bigr).
\end{align}
Applying \eqref{E:Second_Covariant_First} and \eqref{E:Second_Covariant_Second} to \eqref{E:Bochner_Connection} we get
\begin{align*}
  [\Delta_B X]_j = \sum_{i=1}^2\left(\sum_{k=1}^3[P_\Gamma]_{jk}\mathrm{tr}\bigl[(e_i\otimes e_i)\nabla_\Gamma^2 X_k\bigr]+(AX\cdot e_i)[Ae_i]_j\right).
\end{align*}
Furthermore, since $e_1$ and $e_2$ form an orthonormal basis of the tangent plane of $\Gamma$ it follows that $\sum_{i=1}^2(e_i\otimes e_i)=P_\Gamma$ and thus
\begin{align*}
  \sum_{i=1}^2\mathrm{tr}\bigl[(e_i\otimes e_i)\nabla_\Gamma^2X_k\bigr] = \mathrm{tr}[P_\Gamma\nabla_\Gamma^2X_k] = \mathrm{tr}[\nabla_\Gamma^2X_k] = \Delta_\Gamma X_k
\end{align*}
for each $k=1,2,3$ by $P_\Gamma\nabla_\Gamma=\nabla_\Gamma$, and
\begin{align*}
  \sum_{i=1}^2(AX\cdot e_i)Ae_i = \sum_{i=1}^2A(e_i\otimes e_i)AX = AP_\Gamma AX = A^2X,
\end{align*}
by $(AX\cdot e_i)Ae_i=(Ae_i\otimes e_i)AX=A(e_i\otimes e_i)AX$ and \eqref{E:W_OrthProj}.
Therefore,
\begin{align*}
  [\Delta_BX]_j = \sum_{k=1}^3[P_\Gamma]_{jk}\Delta_\Gamma X_k+[A^2X]_j = [P_\Gamma\Delta_\Gamma X]_j+[A^2X]_j
\end{align*}
for each $j=1,2,3$, which yields the formula \eqref{E:Bochner_Beltrami}.

\end{appendices}

\section*{Acknowledgments}
The author is grateful to Professor Yoshikazu Giga for his valuable comments on this work.
The work of the author was supported by Grant-in-Aid for JSPS Fellows No. 16J02664 and the Program for Leading Graduate Schools, MEXT, Japan.

\bibliographystyle{hsiam}
\bibliography{Euler_NS_Thin_Ref_arXiv}

\begin{thebibliography}{10}

\bibitem{Ari89bk}
{\sc R.~Aris}, {\em Vectors, tensors and the basic equations of fluid
  mechanics}, Mineola, NY: Dover Publications, reprint of the 1962
  original~ed., 1989.

\bibitem{A66}
{\sc V.~Arnol$'$d}, {\em Sur la g\'eom\'etrie diff\'erentielle des groupes de
  {L}ie de dimension infinie et ses applications \`a l'hydrodynamique des
  fluides parfaits}, Ann. Inst. Fourier (Grenoble), 16 (1966), pp.~319--361.

\bibitem{A89bk}
{\sc V.~I. Arnol$'$d}, {\em Mathematical methods of classical mechanics},
  vol.~60 of Graduate Texts in Mathematics, Springer-Verlag, New York,
  second~ed., 1989.
\newblock Translated from the Russian by K. Vogtmann and A. Weinstein.

\bibitem{ADe09}
{\sc M.~Arroyo and A.~DeSimone}, {\em Relaxation dynamics of fluid membranes},
  Phys. Rev. E (3), 79 (2009), pp.~031915, 17.

\bibitem{BGN15}
{\sc J.~W. Barrett, H.~Garcke, and R.~N\"urnberg}, {\em Stable numerical
  approximation of two-phase flow with a {B}oussinesq-{S}criven surface fluid},
  Commun. Math. Sci., 13 (2015), pp.~1829--1874.

\bibitem{BP10}
{\sc D.~Bothe and J.~Pr\"uss}, {\em On the two-phase {N}avier-{S}tokes
  equations with {B}oussinesq-{S}criven surface fluid}, J. Math. Fluid Mech.,
  12 (2010), pp.~133--150.

\bibitem{B1913}
{\sc M.~J. Boussinesq}, {\em Sur l'existence d'une viscosit\'e superficielle,
  dans la mince couche de transition s\'eparant un liquide d'un autre fluide
  contigu}, Ann. Chim. Phys., 29 (1913), pp.~349--357.

\bibitem{CFG05}
{\sc P.~Cermelli, E.~Fried, and M.~E. Gurtin}, {\em Transport relations for
  surface integrals arising in the formulation of balance laws for evolving
  fluid interfaces}, J. Fluid Mech., 544 (2005), pp.~339--351.

\bibitem{C15bk}
{\sc B.-Y. Chen}, {\em Total mean curvature and submanifolds of finite type},
  vol.~27 of Series in Pure Mathematics, World Scientific Publishing Co. Pte.
  Ltd., Hackensack, NJ, second~ed., 2015.
\newblock With a foreword by Leopold Verstraelen.

\bibitem{DE07}
{\sc G.~Dziuk and C.~M. Elliott}, {\em Finite elements on evolving surfaces},
  IMA J. Numer. Anal., 27 (2007), pp.~262--292.

\bibitem{DE13}
\leavevmode\vrule height 2pt depth -1.6pt width 23pt, {\em Finite element
  methods for surface {PDE}s}, Acta Numer., 22 (2013), pp.~289--396.

\bibitem{EM70}
{\sc D.~G. Ebin and J.~Marsden}, {\em Groups of diffeomorphisms and the motion
  of an incompressible fluid}, Ann. of Math. (2), 92 (1970), pp.~102--163.

\bibitem{GT01bk}
{\sc D.~Gilbarg and N.~S. Trudinger}, {\em Elliptic partial differential
  equations of second order}, Classics in Mathematics, Springer-Verlag, Berlin,
  2001.
\newblock Reprint of the 1998 edition.

\bibitem{G81bk}
{\sc M.~E. Gurtin}, {\em An introduction to continuum mechanics}, vol.~158 of
  Mathematics in Science and Engineering, Academic Press, Inc. [Harcourt Brace
  Jovanovich, Publishers], New York-London, 1981.

\bibitem{HS10}
{\sc L.~T. Hoang and G.~R. Sell}, {\em Navier-{S}tokes equations with {N}avier
  boundary conditions for an oceanic model}, J. Dynam. Differential Equations,
  22 (2010), pp.~563--616.

\bibitem{IRS07}
{\sc D.~Iftimie, G.~Raugel, and G.~R. Sell}, {\em Navier-{S}tokes equations in
  thin 3{D} domains with {N}avier boundary conditions}, Indiana Univ. Math. J.,
  56 (2007), pp.~1083--1156.

\bibitem{JOR17pre}
{\sc T.~{Jankuhn}, M.~A. {Olshanskii}, and A.~{Reusken}}, {\em {Incompressible
  fluid problems on embedded surfaces: Modeling and variational formulations}},
  ArXiv e-prints,  (2017), 1702.02989.

\bibitem{KLG16}
{\sc H.~Koba, C.~Liu, and Y.~Giga}, {\em Energetic variational approaches for
  incompressible fluid systems on an evolving surface}, Quart. Appl. Math., 75
  (2017), pp.~359--389.

\bibitem{KN96bk}
{\sc S.~Kobayashi and K.~Nomizu}, {\em Foundations of differential geometry.
  {V}ol. {II}}, Wiley Classics Library, John Wiley \& Sons, Inc., New York,
  1996.
\newblock Reprint of the 1969 original, A Wiley-Interscience Publication.

\bibitem{L13bk}
{\sc J.~M. Lee}, {\em Introduction to smooth manifolds}, vol.~218 of Graduate
  Texts in Mathematics, Springer, New York, second~ed., 2013.

\bibitem{MT01}
{\sc M.~Mitrea and M.~Taylor}, {\em Navier-{S}tokes equations on {L}ipschitz
  domains in {R}iemannian manifolds}, Math. Ann., 321 (2001), pp.~955--987.

\bibitem{M17}
{\sc T.-H. Miura}, {\em Zero width limit of the heat equation on moving thin
  domains}, Interfaces Free Bound., 19 (2017), pp.~31--77.

\bibitem{MGL17pre}
{\sc T.-H. Miura, Y.~Giga, and C.~Liu}, {\em An energetic variational approach
  for nonlinear diffusion equations in moving thin domains}, Hokkaido
  University Preprint Series in Math., \#1101 (2017).

\bibitem{N07bk}
{\sc L.~I. Nicolaescu}, {\em Lectures on the geometry of manifolds}, World
  Scientific Publishing Co. Pte. Ltd., Hackensack, NJ, second~ed., 2007.

\bibitem{NVW12}
{\sc I.~Nitschke, A.~Voigt, and J.~Wensch}, {\em A finite element approach to
  incompressible two-phase flow on manifolds}, J. Fluid Mech., 708 (2012),
  pp.~418--438.

\bibitem{P06bk}
{\sc P.~Petersen}, {\em Riemannian geometry}, vol.~171 of Graduate Texts in
  Mathematics, Springer, New York, second~ed., 2006.

\bibitem{PRR02}
{\sc M.~Prizzi, M.~Rinaldi, and K.~P. Rybakowski}, {\em Curved thin domains and
  parabolic equations}, Studia Math., 151 (2002), pp.~109--140.

\bibitem{R95}
{\sc G.~Raugel}, {\em Dynamics of partial differential equations on thin
  domains}, in Dynamical systems ({M}ontecatini {T}erme, 1994), vol.~1609 of
  Lecture Notes in Math., Springer, Berlin, 1995, pp.~208--315.

\bibitem{RS93}
{\sc G.~Raugel and G.~R. Sell}, {\em Navier-{S}tokes equations on thin {$3$}{D}
  domains. {I}. {G}lobal attractors and global regularity of solutions}, J.
  Amer. Math. Soc., 6 (1993), pp.~503--568.

\bibitem{S60}
{\sc L.~Scriven}, {\em Dynamics of a fluid interface equation of motion for
  newtonian surface fluids}, Chemical Engineering Science, 12 (1960),
  pp.~98--108.

\bibitem{T92}
{\sc M.~E. Taylor}, {\em Analysis on {M}orrey spaces and applications to
  {N}avier-{S}tokes and other evolution equations}, Comm. Partial Differential
  Equations, 17 (1992), pp.~1407--1456.

\bibitem{T11bk}
\leavevmode\vrule height 2pt depth -1.6pt width 23pt, {\em Partial differential
  equations {II}. {Q}ualitative studies of linear equations}, vol.~116 of
  Applied Mathematical Sciences, Springer, New York, second~ed., 2011.

\bibitem{TZ96}
{\sc R.~Temam and M.~Ziane}, {\em Navier-{S}tokes equations in
  three-dimensional thin domains with various boundary conditions}, Adv.
  Differential Equations, 1 (1996), pp.~499--546.

\bibitem{TZ97}
\leavevmode\vrule height 2pt depth -1.6pt width 23pt, {\em Navier-{S}tokes
  equations in thin spherical domains}, in Optimization methods in partial
  differential equations ({S}outh {H}adley, {MA}, 1996), vol.~209 of Contemp.
  Math., Amer. Math. Soc., Providence, RI, 1997, pp.~281--314.

\end{thebibliography}

\end{document}